\documentclass[11pt]{amsart}
\issueinfo{00}{0}{Month}{Year}
\copyrightinfo{2008}{University of Calgary}
\pagespan{1}{2}
\PII{ISSN 1715-0868}
\usepackage{graphicx}
\usepackage{epstopdf}

\usepackage{amsthm,amssymb,amsxtra,amscd,amsmath}
\usepackage{verbatim}
\usepackage{multicol}
\usepackage{url}
\usepackage{algorithmic}
\usepackage{multirow}
\newtheorem{theorem}{Theorem}[section]

\newtheorem{lemma}[theorem]{Lemma}
\newtheorem{corollary}[theorem]{Corollary}

\newtheorem{proposition}[theorem]{Proposition}

\newtheorem{example}[theorem]{Example}
\newtheorem*{example*}{Example}

\usepackage{mathtools}
\usepackage{subcaption}
\newtheoremstyle{myexample}{3pt}{3pt}{\rmfamily}{}{\itshape}{:}{ }{\thmname{#1}\thmnumber{ #2}\thmnote{ (#3)}}
\theoremstyle{myexample}

\newtheoremstyle{myremark}{3pt}{3pt}{\rmfamily}{}{\itshape}{:}{ }{\thmname{#1}}
\theoremstyle{myremark}

\newtheorem*{observation*}{Observation}

\newtheoremstyle{conjecture}{3pt}{3pt}{\itshape}{}{\bfseries}{.}{ }{\thmname{#1}\thmnote{ (#3)}}
\theoremstyle{conjecture}

\newtheorem*{question*}{Question}

\newtheorem{theorem*}{Theorem}

\numberwithin{equation}{section}

\usepackage{MnSymbol}
\newcounter{algorithm}
\def \ga {{\gamma}}
\def \ze {{\zeta}}

\def \I {\mathbb{I}}
\def \A {\mathbb{A}}
\def \B {\mathbb{B}}
\def \C {\mathbb{C}}
\def \D {\mathbb{D}}

\usepackage{mathtools}
\usepackage{pgf,tikz}
\usepackage{array}
\tikzstyle{vertex}=[circle,draw, inner sep=0pt, minimum size=4pt] 
\newcommand{\vertex}{\node[vertex]}
\usetikzlibrary{decorations.markings,shapes,positioning}

\begin{document}
\include{cdmstdcmds}
\title{On the double covers of a line graph}
\author{Shivani Chauhan}
\address{Department of 
Mathematics, Shiv Nadar 
Institution of Eminence, Dadri, U.P 201314, India} 
\email{sc739@snu.edu.in}

\author{A. Satyanarayana Reddy}
\address{Department of 
Mathematics, Shiv Nadar 
Institution of Eminence, Dadri, U.P 201314, India}
\email{satya.a@snu.edu.in}
\date{(date1), and in revised form (date2).}
\thanks{The authors would like to thank the handling editor and the anonymous reviewers for their careful reading of the manuscript.}
\begin{abstract}
Let $L(X)$ be the line graph of graph $X$. Let $X^{\prime\prime}$ be the Kronecker product of $X$ by $K_2$. In this paper, we see that $L(X^{\prime\prime})$ is a double cover of $L(X)$.
We define the symmetric edge graph of $X$, denoted as 
$\ga(X)$ which is also a double cover of $L(X)$. We study various properties of $\ga(X)$ in relation to $X$ and the relationship amongst the three double covers of $L(X)$ that are $L(X^{\prime\prime}),\ga(X)$ and $L(X)^{\prime\prime}$. With the help of these double covers, we show that for any integer $k\geq 5$, there exist two equienergetic graphs of order $2k$ that are not cospectral.
\end{abstract}
\maketitle
\section{Introduction}
In this paper, we restrict ourselves to finite graphs with no self-loops and multiple edges. We denote the {\em cycle graph, the path graph, the complete graph} and the {\em star graph} on $n$ vertices by $C_n, P_n, K_n$ and $K_{1,n-1}$ respectively. A graph $Y$ is a {\em covering graph} of a graph $X$ if there is a map from the vertex set of $Y$ to the vertex set of $X$ such that
the neighbourhood of a vertex $v$ in $Y$ is mapped bijectively onto the neighbourhood of $f(v)$ in $X$. If each vertex of $X$ has exactly two preimages in $Y$ then we say that $Y$ is a {\em double cover} of $X$. One of the easy ways to construct the double cover of a graph $X$ is to take the {\em Kronecker product} of $X$ by $K_2$ and it is denoted by $X^{\prime\prime}$. The Kronecker product $X_1\times X_2$ of graphs $X_1$ and $X_2$ is a graph such that the vertex set is $V(X_1)\times V(X_2)$, vertices $(x_1,x_2)$ and $(x_1^\prime,x_2^\prime)$ are adjacent if and only if $x_1$ is adjacent to $x_1^\prime$ in $X_1$ and $x_2$ is adjacent to $x_2^\prime$ in $X_2.$ In Section \ref{thm:Sec:Main:3}, we show that $L(X^{\prime\prime})$ is a double cover of $L(X)$. The double cover for a graph $X$ is not unique (see Example \ref{ex:coverK4}). Many researchers used covering of graphs in the construction of Ramanujan graphs (see \cite{minei2018ramanujan}) and in the construction of pairs of cospectral but not isomorphic graphs. Additional information on covering of graphs can be found in \cite{gross1977generating,terras2010zeta}.
\par
\begin{example}\label{ex:coverK4}
In this example, we demonstrate the two non-isomorphic double covers of $K_4$.
\begin{figure}[ht!]
    \centering

\begin{tabular}{ccc}
    \begin{tikzpicture}   
    \vertex (1) at (0,0) [label=below:$a$]{};
    \vertex (2) at (2,0) [label=below:$b$]{};
    \vertex (3) at (2,2) [label=above:$c$]{};
    \vertex (4) at (0,2) [label=above:$d$]{};
     \path[-]
    (1) edge (2)
    (2) edge (3)
    (4) edge (3)
    (1) edge (4)
    (4) edge (2)
    (1) edge (3)
    ;
     \end{tikzpicture}
     & 
     \begin{tikzpicture}
    \vertex (1) at (1,1)[label=left:$b^{\prime\prime}$] {};
    \vertex (2) at (2,2) [label=above:$c^{\prime\prime}$] {} {};
    \vertex (3) at (2,0) [label=below:$d^{\prime\prime}$] {};
    \vertex (4) at (3,2)[label=above:$d^{\prime}$] {};
    \vertex (5) at (3,0) [label=below:$c^{\prime}$] {};
    \vertex (6) at (4,1) [label=right:$b^{\prime}$] {};
    \vertex (7) at (2,1) [label=right:$a^{\prime\prime}$]{};
    \vertex (8) at (3,1) [label=above:$a^{\prime}$]{};
    
    \path[-]
    (1) edge (2)
    (1) edge (3)
    (2) edge (4)
    (3) edge (5)
    (4) edge (6)
    (5) edge (6)
    (7) edge (2)
    (7) edge (3)
    (7) edge (1)
    (8) edge (4)
    (8) edge (6)
    (8) edge (5);
    
\end{tikzpicture}
     & 
      \begin{tikzpicture}   
    \vertex (1) at (0,0) [label=below:$a^{\prime\prime}$]{};
    \vertex (2) at (3,0) [label=below:$c^{\prime}$]{};
    \vertex (3) at (3,3) [label=above:$d^{\prime\prime}$]{};
    \vertex (4) at (0,3) [label=above:$b^{\prime\prime}$]{};
     \vertex (5) at (1,1) [label=below:$d^\prime$]{};
    \vertex (6) at (2,1) [label=below:$b^\prime$]{};
    \vertex (7) at (2,2) [label=above:$a^\prime$]{};
    \vertex (8) at (1,2) [label=above:$c^{\prime\prime}$]{};
     \path[-]
    (1) edge (2)
    (2) edge (3)
    (4) edge (3)
    (1) edge (4)
    (5) edge (6)
    (6) edge (7)
    (7) edge (8)
    (5) edge (8)
    (4) edge (8)
    (7) edge (3)
    (2) edge (6)
    (1) edge (5)
    ;
     \end{tikzpicture}
     \\
\end{tabular}
\caption{}
    \label{fig:my_label}
\end{figure}
\end{example}  

 Let $X=(V,E)$ be a graph with
 $|V(X)|=n, \;|E(X)|=m.$ We orient the edges arbitrarily and label them as $e_1,e_2,\ldots,e_m$ and also
$e_{m+i}=e_i^{-1},\;1\le i\le m,$ where $e_k^{-1}$ denotes the edge $e_k$
with the direction reversed. Then the {\em edge adjacency matrix} of $X$, denoted by $M(X)$ or simply $M$, is defined as 
 $$M_{ij}=\begin{cases}
          1 & \mbox{if $t(e_i)=s(e_j)$\; and\; $s(e_i)\ne t(e_j)$},\\
          0 & \mbox{otherwise.}
                   \end{cases}$$
where $s(e_i)$ and $t(e_i)$ denote the starting and terminal vertex of $e_i$ respectively.
\begin{example}
The process of computation of matrix $M(C_3)$ is given below.
\vglue 2mm
\begin{tabular}{|c|c|c|}
\hline
$X$ && $M$\\
\begin{tikzpicture}
\vertex (1) at (0,0){};
\vertex (2) at (2,0){};
\vertex (3) at (1,1){};
     \path[-]
    (1) edge (2)
    (2) edge (3)
    (1) edge (3)
    ; 
\end{tikzpicture}

&
\tikzset{every picture/.style={line width=0.75pt}} 

\begin{tikzpicture}[x=0.75pt,y=0.75pt,yscale=-0.5,xscale=0.5]

\draw   (100,115) .. controls (100,112.24) and (102.24,110) .. (105,110) .. controls (107.76,110) and (110,112.24) .. (110,115) .. controls (110,117.76) and (107.76,120) .. (105,120) .. controls (102.24,120) and (100,117.76) .. (100,115) -- cycle ;
\draw    (100.48,120.21) -- (41.42,199.18) ;
\draw    (110,119) -- (112.48,122.84) -- (172,203.5) ;
\draw    (44.61,202.3) -- (167.5,208) ;

\draw   (34,202.3) .. controls (34,205.17) and (36.37,207.5) .. (39.3,207.5) .. controls (42.23,207.5) and (44.61,205.17) .. (44.61,202.3) .. controls (44.61,199.43) and (42.23,197.1) .. (39.3,197.1) .. controls (36.37,197.1) and (34,199.43) .. (34,202.3) -- cycle ;
\draw   (167.5,208) .. controls (167.5,210.49) and (169.51,212.5) .. (172,212.5) .. controls (174.49,212.5) and (176.5,210.49) .. (176.5,208) .. controls (176.5,205.51) and (174.49,203.5) .. (172,203.5) .. controls (169.51,203.5) and (167.5,205.51) .. (167.5,208) -- cycle ;
\draw   (126,129.5) -- (124.67,138.67) -- (115.88,135.76) ;
\draw   (151,165.5) -- (141.18,160.84) -- (140.52,171.69) ;
\draw   (63,195.5) -- (75,202.5) -- (63,209.5) ;
\draw   (128.84,211.14) -- (118.63,205.82) -- (128.49,199.88) ;
\draw   (95,142.5) -- (83.8,143.58) -- (83.41,132.33) ;
\draw   (52.61,173.37) -- (62.62,170.31) -- (65,180.5) ;

\draw (128,114) node [anchor=north west][inner sep=0.75pt]   [align=left] {$e_1$};
\draw (154,151) node [anchor=north west][inner sep=0.75pt]   [align=left] {$e_1^{-1}$};
\draw (116,179) node [anchor=north west][inner sep=0.75pt]   [align=left] {$e_2$};
\draw (59,210) node [anchor=north west][inner sep=0.75pt]   [align=left] {$e_2^{-1}$};
\draw (30,145) node [anchor=north west][inner sep=0.75pt]   [align=left] {$e_3$};
\draw (61,92) node [anchor=north west][inner sep=0.75pt]   [align=left] {$e_3^{-1}$};
\end{tikzpicture}
&
\begin{tabular}{|l|lll|lll|}
\hline
&$e_1$&$e_2$&$e_3$&$e_1^{-1}$&$e_2^{-1}$&$e_3^{-1}$\\
\hline
$e_1$&0 & 1 & 0 & 0 & 0 & 0\\
$e_2$&0 & 0 & 1 & 0 & 0 & 0\\
$e_3$&1 & 0 & 0 & 0 & 0 & 0\\
\hline
$e_1^{-1}$&0 & 0 & 0 & 0 & 0 & 1\\
$e_2^{-1}$&0 & 0 & 0 & 1 & 0 & 0\\
$e_3^{-1}$&0 & 0 & 0 & 0 & 1 & 0\\
\hline
\end{tabular}\\
&&\\
\hline
\end{tabular}
\end{example}

It is interesting to see that $M+M^T,$ where $A^T$ denotes the transpose of $A,$ is a symmetric matrix with entries $0$ or $1.$ We call $M+M^T$ {\em symmetric edge adjacency matrix} of $X,$ and the graph whose adjacency matrix is $M+M^T$ is called {\em symmetric edge graph} of $X$, denoted by $\ga(X)$.  We define $\ga^{k}(X)=\ga(\ga^{k-1}(X)),$ where $k\in \mathbb{N}$ with $\ga^0(X)=X$. Later, we will see that $\ga(X)$ is also a double cover of $L(X)$. In Figure \ref{fig:3cover}, for a graph $X$ we have given its line graph and the three non-isomorphic double covers of $L(X)$.

\begin{figure}[ht!]
\begin{subfigure}[h]{0.1\textwidth}
\begin{tikzpicture} [xscale=0.6,yscale=0.9]  
    \vertex (1) at (1,1) [label=below:$3$]{};
    \vertex (2) at (2,1)[label=below:$2$] {};
    \vertex (3) at (1.5,2)[label=right:$1$]{};
    \vertex (6) at (0,1) [label=below:$6$]{};
    \vertex (4) at (1.5,3)[label=left:$5$]{};
    \vertex (5) at (3,1)[label=below:$4$]{};
     \path[-]
    (1) edge (2)
    (2) edge (3)
    (1) edge (3)
    (1) edge (6)
    (2) edge (5)
    (3) edge (4)
     ;
\end{tikzpicture}
\caption{$X$}
\label{fig:X}
\end{subfigure}
\hfill
\begin{subfigure}[ht!]{0.1\textwidth}
\begin{tikzpicture}
    \vertex (1) at (0,0){};
    \vertex (2) at (2,0){};
    \vertex (3) at (1,1){};
    \vertex (4) at (0.5,0.5){};
    \vertex (5) at (1.5,0.5){};
    \vertex (6) at (1,0){};
\path[-]
(1) edge (4)
(3) edge (4)
(6) edge (5)
(3) edge (5)
(1) edge (6)
(6) edge (2)
(4) edge (6)
(2) edge (5)
(5) edge (4)
;   
\end{tikzpicture}
\caption{$L(X)$}
\label{fig:L(X)}
\end{subfigure}
\hfill
\begin{subfigure}[ht!]{0.2\textwidth}
    \begin{tikzpicture}[xscale=0.5,yscale=0.5]
    \vertex (1) at (2,3) {};
    \vertex (2) at (3,1) {};
    \vertex (3) at (1,1) {};
    \vertex (4) at (1,2) {};
    \vertex (5) at (3,2) {};
    \vertex (6) at (2,0) {};
    \vertex (7) at (3.75,3.75){};
    \vertex (8) at (4.5,1.5) {};
    \vertex (9) at (4.5,-1) {};
    \vertex (10) at (0.5,-1){};
    \vertex (12) at (0,3.75) {};
    \vertex (11) at (-1,1.5) {};
    \path[-]
    (1) edge (3)
    (3) edge (2)
    (2) edge (1)
    (6) edge (5)
    (4) edge (5)
    (4) edge (6)
    (1) edge (7)
    (5) edge (7)
    (8) edge (5)
    (2) edge (8)
    (2) edge (9)
    (9) edge (6)
    (3) edge (10)
    (10) edge (6)
    (4) edge (12)
    (1) edge (12)
    (4) edge (11)
    (3) edge (11)
    ;
\end{tikzpicture}
        \caption{$\ga(X)$}
\end{subfigure}
\hfill
\begin{subfigure}[ht!]{0.1\textwidth}
\begin{tikzpicture}[xscale=0.6,yscale=0.5]
    \vertex (1) at (1,1) {};
    \vertex (2) at (2,2) {};
    \vertex (3) at (2,0) {};
    \vertex (4) at (3,2) {};
    \vertex (5) at (3,0) {};
    \vertex (6) at (4,1) {};
    \vertex (7) at (1,3) {};
    \vertex (8) at (2.5,4){};
    \vertex (9) at (4,3) {};
    \vertex (10) at (4,-1) {};
    \vertex (11) at (2.5,-2) {};
    \vertex (12) at (1,-1) {};
    
    \path[-]
    (1) edge (2)
    (1) edge (3)
    (2) edge (4)
    (3) edge (5)
    (4) edge (6)
    (5) edge (6)
    (1) edge (7)
    (2) edge (7)
    (8) edge (4)
    (8) edge (2)
    (9) edge (4)
    (9) edge (6)
    (10) edge (5)
    (10) edge (6)
    (11) edge (3)
    (11) edge (5)
    (12) edge (1)
    (12) edge (3)
    ;
\end{tikzpicture}
\caption{$L(X^{\prime\prime})$}
    \end{subfigure}
    \hfill
    \begin{subfigure}[ht!]{0.1\textwidth}
\begin{tikzpicture}[xscale=0.5,yscale=0.9]
    \vertex (1) at (0,0) {};
    \vertex (2) at (1,0) {};
    \vertex (3) at (2,0) {};
    \vertex (4) at (0,1) {};
    \vertex (5) at (1,1) {};
    \vertex (6) at (2,1) {};
    \vertex (7) at (0,2) {};
    \vertex (8) at (1,2) {};
    \vertex (9) at (2,2) {};
    \vertex (10) at (0,3) {};
    \vertex (11) at (1,3) {};
    \vertex (12) at (2,3) {};
    
    \path[-]
    (1) edge (4)
    (1) edge (5)
    (2) edge (4)
    (2) edge (6)
    (3) edge (6)
    (3) edge (5)
    (7) edge (4)
    (8) edge (4)
    (5) edge (7)
    (5) edge (9)
    (8) edge (6)
    (9) edge (6)
    (7) edge (10)
    (7) edge (11)
    (8) edge (10)
    (8) edge (12)
    (9) edge (11)
    (9) edge (12)
    ;
\end{tikzpicture}
\caption{$L(X)^{\prime\prime}$}
    \end{subfigure}
    \caption{}
    \label{fig:3cover}
\end{figure}

In the literature, a lot of work has been done on the properties of $L(X)$ in relation to $X$ (see Chapter 8 of \cite{Harary1969}). 
In Section~\ref{sec:ga}, we study various properties of $\ga(X)$ with respect to $X$. We provide a decomposition of $\ga(X)$ in terms of crown graphs.
With these three double covers of $L(X)$ in hand which are $L(X)^{\prime\prime},L(X^{\prime\prime})$ and $\ga(X)$, we will study the relation amongst them in Section \ref{sec:DC}. In Theorem \ref{thm:Sec:Main}, we characterize all graphs $X$ so that $\ga(X)=L(X^{\prime\prime}),\ga(X)=L(X)^{\prime\prime}$ and $L(X)^{\prime\prime}=L(X^{\prime\prime})$. In the rest of this section, we will discuss why the matrix $M$ is important for the Ihara zeta function of a graph, the properties of the matrix $M$, and the symmetric edge graphs.

A path $P=e_1e_2\cdots e_t,$ where $e_i$ is an oriented edge, is said to {\em backtrack} if $e_{k+1}=e_k^{-1}$ for some $k\in \{1,2,3,\ldots,t-1\},$ {\it i.e.} it crosses the same edge twice in a row. A path $P$ is said to have a {\em tail} if $e_t=e_1^{-1},$ {\it i.e.} the last edge of $P$ is the reverse of the first edge. A closed path $C=e_1e_2\cdots e_t$ is said to be {\em prime} or {\em  
primitive} if it has no backtrack or tail and $C\ne D^f$ for some closed path $D$ and $f>1.$  The {\em Ihara zeta function} of a graph $X$ is defined to be 
$$\ze_{X}(u)=\prod\limits_{[C]}\left(1-u^{\ell(C)}\right)^{-1},$$ where the
product is over the primes $[C]$ of $X$ and $\ell(C)$ is the length of cycle
$C$. The {\em fundamental group} $\pi_1(X,v)$ of a connected graph $X$ is the free group consisting of all closed walks starting and ending at the vertex $v$ together with the operation which concatenates walks. The rank $r$ of the $\pi_1(X,v)$ is the number of elements in a minimal generating set of $\pi_1(X,v)$ which is also the number of edges left out of a spanning tree of $X.$
The computation of Ihara zeta function using the definition is difficult except for the cycle graph. The following two results by Bass~\cite{bass1992ihara} and Hashimoto~\cite{MR1040609} simplified the evaluation of the Ihara zeta function for graphs that have a minimal degree of at least 2. 

\begin{theorem}\cite{MR1040609}\label{thm:AQ}
  Let $A(X)$ or $A$ be the adjacency matrix of $X$ and $Q(X)$ or $Q$ be the diagonal matrix with $j^{th}$ diagonal entry $q_j$ such that $q_j + 1$ is the degree of the $j^{th}$ vertex of X. Suppose that r is the rank of the fundamental group of $X$; $r-1=|E|-|V|$. Then
 \begin{equation}\label{eqn:1}
\ze_{X}(u)^{-1}=(1-u^2)^{r-1}det(I-Au+Qu^2).
\end{equation}
 \end{theorem} 
  The main purpose of introducing the matrix $M$ can be seen in the following result. 
\begin{theorem} \cite{bass1992ihara}\label{Thm:IHARAM}
Let $M$ be the edge adjacency matrix of a graph $X.$ Then
 $$\ze_{X}(u)^{-1}=det(I-Mu).$$
 \end{theorem}

 Now we will state a few properties of matrix $M$.  Many of these have been discussed in the thesis of Horton~\cite{horton2006ihara} and one can also find them in the book by Terras~\cite{terras2010zeta}. 
 \begin{equation}\label{eqn :2}
M=\left[\begin{array}{c|c}
  \A & \B\\
  \hline
  \C & \D
\end{array}
\right],
\end{equation}
 where $\A, \B,\C,\D$ are $m\times m$ matrices with the following properties:
 \begin{multicols}{2}
 \begin{enumerate}
     \item $\B=\B^T$, $\C=\C^T.$
     \item $\D=\A^T.$
     \item The diagonals of $\A,\B,\C$ and $\D$ are zeros.
     \item If $J=\begin{bmatrix}
0 & \I_{m}\\
\I_{m} & 0
\end{bmatrix},$ where $\I_m$ denotes the identity matrix of order $m$ then $M^T=JMJ.$
\item The $i^{th}$ row sum of $M$ is equal to $d_{t(e_i)}-1,$ where $d_{v}$ denotes the degree of vertex $v.$

\item \label{property:M:6}
The sum of the blocks of $M,$ $\A+\B+\C+\D$ is the adjacency matrix of  $L(X).$ Here one can note that Hadamard product of any two matrices from $\{\A,\B,\C,\D\}$ is the zero matrix.

\item \label{property:M:7}
Let $M$ be the edge adjacency matrix of the graph $X.$ Then $Tr(M^k)=N_k,$ where $N_k$ is the number of cycles of length $k$ without backtracks and tails.
\end{enumerate}
\end{multicols}

Now we provide two examples of $\ga(X)$, from where one can note that $\ga$ function does not preserve connectivity and $K_{1,3}$ is a tree but $\ga(K_{1,3})$ is a cycle graph. After that, we shall state Theorem \ref{thm:unionspectra}, which is essential for further discussion.

\begin{example}\label{Ex:cycle}
\begin{enumerate}

\item \label{Ex:cycle:1}
If $X=C_n,$ then $\ga(X)=2C_n$ and $\ga^{k}(X)=2^{k}C_n$.
\begin{figure}[ht!]
\centering
\begin{tabular}{ccc}
\begin{tikzpicture}
    \vertex (1) at (2,2) {};
    \vertex (2) at (3,3) {};
    \vertex (3) at (4,2) {};
    \path[-]
    (1) edge (2)
    (2) edge (3)
    (3) edge (1);
    
\end{tikzpicture}
&
\begin{tikzpicture}
 \vertex (4) at (0,0) {};
    \vertex (5) at (2,0) {};
    \vertex (6) at (1,1) {};
    \vertex (7) at (0,-1.5){};
    \vertex (8) at (2,-1.5){};
    \vertex (9) at (1,-0.5){};
    \path[-]
    (4) edge (5)
    (5) edge (6)
    (6) edge (4)
    (7) edge (8)
    (8) edge (9)
    (9) edge (7);
\end{tikzpicture}

&
\begin{tabular}{|l|lll|lll|}
\hline
&$e_1$&$e_2$&$e_3$&$e_1^{-1}$&$e_2^{-1}$&$e_3^{-1}$\\
\hline
$e_1$&0 & 1 & 1 & 0 & 0 & 0\\
$e_2$&1 & 0 & 1 & 0 & 0 & 0\\
$e_3$&1 & 1 & 0 & 0 & 0 & 0\\

\hline
$e_1^{-1}$&0 & 0 & 0 & 0 & 1 & 1\\
$e_2^{-1}$&0 & 0 & 0 & 1 & 0 & 1\\
$e_3^{-1}$&0 & 0 & 0 & 1 & 1 & 0\\
\hline
\end{tabular}\\
\\
\end{tabular}
\caption{$C_3,\ga(C_3)$ and $A(\ga(C_3)).$}
\label{fig: Cycle}
\end{figure}

\item \label{Ex:cycle:2} If $X=K_{1,3},$ then $\ga(X)=C_6$ and $\ga^{k}(X)=2^{k-1}C_6$.
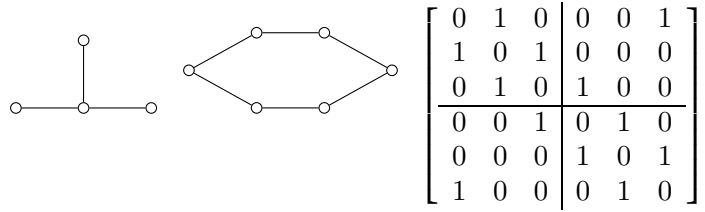
\begin{figure}[ht!]
\begin{center}  
\begin{tabular}{ccc}
\begin{tikzpicture}[xscale=0.9,yscale=0.9]
    \vertex (1) at (3,3) {};
    \vertex (2) at (4,3) {};
    \vertex (3) at (2,3) {};
    \vertex (4) at (3,4) {};
     \path[-]
    (1) edge (2)
    (1) edge (3)
    (1) edge (4)
    ;
\end{tikzpicture}&

\begin{tikzpicture}[xscale=0.9,yscale=0.5]
    \vertex (1) at (1,1) {};
    \vertex (2) at (2,2) {};
    \vertex (3) at (2,0) {};
    \vertex (4) at (3,2) {};
    \vertex (5) at (3,0) {};
    \vertex (6) at (4,1) {};
    ;
    \path[-]
    (1) edge (2)
    (1) edge (3)
    (2) edge (4)
    (3) edge (5)
    (4) edge (6)
    (5) edge (6);
\end{tikzpicture}
&
$\left[\begin{array}{ccc|ccc}
0 & 1 & 0 & 0 & 0 & 1 \\
1 & 0 & 1 & 0 & 0 & 0 \\
0 & 1 & 0 & 1 & 0 & 0 \\
\hline
0 & 0 & 1 & 0 & 1 & 0 \\
0 & 0 & 0 & 1 & 0 & 1 \\
1 & 0 & 0 & 0 & 1 & 0
\end{array}
\right].$
\\
\end{tabular}
\end{center}
\caption{$K_{1,3},\ga(K_{1,3})$ and $A(\ga(K_{1,3})).$}
    \label{fig:K_{1,3}}
    \end{figure}
\end{enumerate}
\end{example}
    
 \begin{theorem}\cite{davis2013circulant}\label{thm:unionspectra}
 Let $H=\begin{bmatrix}
A^\prime & B^\prime\\
B^\prime & A^\prime
\end{bmatrix}$ be a symmetric $2\times 2$ block matrix, where $A^\prime$ and $B^\prime$ are square matrices of same order. Then the spectrum of $H$ is the union of the spectra of $A^\prime+B^\prime$
and $A^\prime-B^\prime$.
\end{theorem}
  
From Equation~\ref{eqn :2}, we have 
\begin{equation}\label{eqn:M+M^t}
  M+M^T=\left[\begin{array}{c|c}
  \A+\D & \B+\C\\
  \hline
  \B+\C & \A+\D
\end{array}
\right].
\end{equation} From Property \ref{property:M:6} of $M$, we see that the row sum of Equation \ref{eqn:M+M^t} is equal to $A(L(X))$, note that $\ga(X)$ is the double cover of $L(X)$. By Theorem \ref{thm:unionspectra}, we can see that the spectrum of $A(L(X))$ is contained in the spectrum of $A(\ga(X))$. The following are a few immediate observations of the graph $\ga(X)$. 
\begin{enumerate}
    \item The number of vertices of $\ga(X)$ is twice the number of edges of $X.$
    \item We have, $$Tr\left((M+M^T)^2\right) = 2|E(\ga(X))|=2e^TMe,$$
 where $e$ denotes the column vector with all entries one
  and $J(M+M^T)=(M+M^T)J.$
  
\item Note that
\begin{equation}\label{eqn:edge}
|E(\ga(X))| = 2|E(L(X))| = \sum\limits_{i=1}^{|V(X)|}d_i^2 - 2|E(X)|.
\end{equation}

\item It is easy to see that if $X$ is Eulerian, then $\ga(X)$ is Eulerian provided $\ga(X)$ is connected which follows from the fact that if $X$ is Eulerian then $L(X)$ is Eulerian (see Harary~\cite{Harary1969}). But if $\ga(X)$ is Eulerian, then $X$ need not be Eulerian which is clear from Part \ref{Ex:cycle:2} of Example \ref{Ex:cycle}.

\item It is well known that if $X$ is regular, then $L(X)$ is regular. This shows that the map $\ga$ maps regular graphs to regular graphs. Conversely, if $\ga(X)$ is regular, then $X$ is either a
regular graph or a semi-regular bipartite graph. It can be seen from Lemma $6.2$ in \cite{ray1967characterization}.
\end{enumerate}

 For further information on the matrix $M$ and the Ihara zeta function, one can refer to \cite{terras2010zeta}. For other results and proofs related to graph theory, we refer to~\cite{Harary1969,biggs1993algebraic}. We recall once again that $L(X)$ and $X^{\prime\prime}$ denote the line graph and Kronecker double cover of $X$, respectively.

\section{Properties of $\ga(X)$}\label{sec:ga}
 We begin this section by stating the famous Whitney theorem and then we present the analogous result for the $\ga$ function.
 
\begin{theorem}\label{Thm:Whitney}\cite{whitney1992congruent}
Let $X$ and $Y$ be connected graphs with isomorphic line graphs. Then $X$ and $Y$ are isomorphic, unless one is $K_3$ and the other is $K_{1,3}$.
\end{theorem}

\begin{theorem} \label{thm:iso}
Let $X$ and $Y$ be connected graphs. Then $\ga(X)$ is isomorphic to $\ga(Y)$ if and only if $X$ is isomorphic to $Y.$
\end{theorem}
\begin{proof}
Suppose that $\ga(X)$ is isomorphic to $\ga(Y)$, then by Property \ref{property:M:6} of $M$  we note that $L(X)$ is isomorphic to $L(Y)$. By Theorem \ref{Thm:Whitney} and Part \ref{Ex:cycle:2} of Example \ref{Ex:cycle}, the result follows.
\end{proof}

Next, we prove that the $\ga$ function is additive with respect to the disjoint union.
\begin{lemma}
Let $X$ be a graph with connected components $X_1,X_2,\ldots,X_k$ {\it i.e.,} $X=X_1\cupdot X_2\cupdot\ldots\cupdot X_k$. Then $$\ga(X_1\cupdot X_2\cupdot \ldots \cupdot X_k)\cong \ga(X_1)\cupdot \ga(X_2)\cupdot \ldots \cupdot \ga(X_k).$$
\end{lemma}

\begin{proof}
 We give the proof for $k=2$ and the general case follows by induction on $k.$ Let $X_1,X_2$ be graphs with $m_1,m_2$ edges, respectively. Then $A(\ga(X_1 \cupdot X_2))$ and $A(\ga(X_1)\cupdot\ga(X_2))$ have the following block structures, respectively. 
 
 $$A(\ga(X_1\cupdot X_2))=\begin{pmatrix}
 A_1 & 0 & B_1 & 0\\
 0 & A_2 & 0 & B_2\\
 B_1 & 0 & A_1 & 0\\
 0 & B_2 & 0 & A_2
 \end{pmatrix}$$
 
 $$A(\ga(X_1)\cupdot\ga(X_2))=
\begin{pmatrix}
A(\ga(X_1)) & 0\\
 0 & A(\ga(X_2))
 \end{pmatrix}
 = \begin{pmatrix}
 A_1 & B_1 & 0 & 0\\
 B_1 & A_1 & 0 & 0\\
 0 & 0 & A_2 & B_2 \\
 0 & 0 & B_2 & A_2
 \end{pmatrix}.$$
 It is easy to see that $$P^TA(\ga(X_1\cupdot X_2))P=A(\ga(X_1)\cupdot\ga(X_2)),$$ where $P=\begin{pmatrix}
 \I_{m_1} & 0 & 0 & 0\\
 0 & 0 & \I_{m_2} & 0\\
 0 & \I_{m_1} & 0 & 0\\
 0 & 0 & 0 & \I_{m_2}
 \end{pmatrix}$ is a permutation matrix.
 \end{proof}

We will see a few examples to observe the pattern of graphs under the $\ga$ function.  For more examples, one can refer the Table~\ref{table:more}.

\begin{example} \label{ex:path}
\begin{enumerate}

 \item \label{ex:path:1} If $X=P_n$ then $\ga^{n-1} (X)$ is a null graph. Table \ref{tab:path} shows the effect of repeated application of the $\ga$ function on the path graph.
 \begin{table}[ht!]
     \centering
     \begin{tabular}{|c|c|c|c|}
    \hline
$X$ & $\ga(X)$ & $\ga ^2(X)$ & $\ga^3(X)$\\
\hline
\begin{tikzpicture}
\vertex (1) at (1,2) {};
\vertex (2) at (1.5,2) {};
\vertex (3) at (2,2) {};
\vertex (4) at (2.5,2) {};
    
    \path[-]
    (1) edge (2)
    (2) edge (3)
    (3) edge (4)
    ;
\end{tikzpicture}
&
\begin{tikzpicture}
 \vertex (1) at (1,2) {};
    \vertex (2) at (1.5,2) {};
    \vertex (3) at (2,2) {};
    \vertex (5) at (2.5,2) {};
    \vertex (6) at (3,2) {};
    \vertex (7) at (3.5,2) {};
    
    \path[-]
    (1) edge (2)
    (2) edge (3)
    (5) edge (6)
    (6) edge (7)
    ;
\end{tikzpicture}&

\begin{tikzpicture}
   \vertex (1) at (1,2) {};
    \vertex (2) at (1.5,2) {};
    \vertex (3) at (2,2) {};
    \vertex (4) at (2.5,2) {};
    \vertex (5) at (3,2) {};
    \vertex (6) at (3.5,2) {};
    \vertex (7) at (4,2) {};
    \vertex (8) at (4.5,2) {};
    \path[-]
    (1) edge (2)
    (3) edge (4)
    (5) edge (6)
    (7) edge (8);

\end{tikzpicture}&

\begin{tikzpicture}
\vertex (1) at (1,2) {};
\vertex (2) at (1.5,2) {};
\vertex (3) at (2,2) {};
\vertex (4) at (2.5,2) {};
\vertex (5) at (3,2) {};
\vertex (6) at (3.5,2) {};
\vertex (7) at (4,2) {};
\vertex (8) at (4.5,2) {};
\end{tikzpicture}\\
\hline
\end{tabular}
     \caption{}
     \label{tab:path}
 \end{table}
 
 \item \label{ex:path:2}
 If $X=K_{1,n},$ then $\ga(X)$ is a crown graph on the $2n$ vertices. In particular, if $X=K_{1,4}$ then $\ga(X)$ is a cube. Recall that a crown graph on $2n$ vertices is a graph with two sets of vertices $\{v_1,v_2,\ldots,v_n\}$ and $\{v_1',v_2',\ldots,v_n'\}$, with an edge from $v_i$ to $v_j^\prime$ whenever $i\neq j.$
 
 \item If $X=K_{2,3}$, then $\ga(X)$ is a $6$-prism graph.
 
 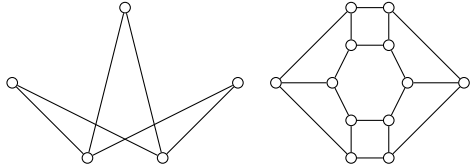
\begin{figure}[ht!]
\begin{center}
\begin{tabular}{cc}
    \begin{tikzpicture}
    \vertex (1) at (1,0) {};
    \vertex (2) at (0,1) {};
    \vertex (3) at (1.5,2) {};
    \vertex (4) at (3,1) {};
    \vertex (5) at (2,0) {};
    \path[-]
    (1) edge (2)
    (2) edge (5)
    (3) edge (1)
    (3) edge (5)
    (4) edge (1)
    (4) edge (5)
    ;
\end{tikzpicture}
&
 \begin{tikzpicture}[xscale=0.5,yscale=0.5]
    \vertex (1) at (3,4) {};
    \vertex (2) at (4,4) {};
    \vertex (3) at (6,2) {};
    \vertex (4) at (4,0) {};
    \vertex (5) at (3,0) {};
    \vertex (6) at (1,2) {};
    \vertex (7) at (3,3) {};
    \vertex (8) at (4,3) {};
    \vertex (9) at (4.5,2) {};
    \vertex (10) at (4,1) {};
    \vertex (11) at (3,1) {};
    \vertex (12) at (2.5,2) {};
   
    \path[-]
    (1) edge (2)
    (2) edge (3)
    (3) edge (4)
    (4) edge (5)
    (5) edge (6)
    (6) edge (1)
    (7) edge (8)
    (8) edge (9)
    (9) edge (10)
    (10) edge (11)
    (11) edge (12)
     (7) edge (12)
    (1) edge (7)
    (2) edge (8)
    (3) edge (9)
    (4) edge (10)
    (5) edge (11)
    (6) edge (12)
     ;
\end{tikzpicture}\\
\end{tabular}
\end{center}
\caption{$K_{2,3}$ and $\ga(K_{2,3}).$}
    \label{fig:K_{2,3}}
  \end{figure}
\end{enumerate}
\end{example}

The following results provide how the $\ga$ function preserves connectedness and bipartiteness. Unless specified otherwise, we assume that $A(\ga(X))=\left[\begin{array}{c|c}
  \A+\D & \B+\C\\
  \hline
  \B+\C & \A+\D
\end{array}
\right]$ and  $A_0=\A+\D, B_0=\B+\C.$

\begin{proposition} \label{pro:conn}
\begin{enumerate} 
 \item \label{pro:conn:1} Let $X$ be a connected graph. Then $\ga (X)$ is connected if and only if $X$ is not a cycle graph or a path graph.
Moreover, $\ga(X)$ cannot be a cycle graph unless $X=K_{1,3}.$

\item  \label{pro:conn:3} Let $\ga(X)$ be a connected graph, then $\ga (X)$ has a cut edge if and only if $X$ contains a pendant vertex which is adjacent to a vertex of degree two.

\item \label{pro:conn:5}Let $X$ be a connected graph. Then $X$ is bipartite if and only if $\ga(X)$ is bipartite.
\end{enumerate}
\end{proposition}

\begin{proof}
Proof of Part~\ref{pro:conn:1}.
Let us suppose that $\ga (X)$ is not a connected graph. Then $B + C = 0$ and hence $B,C=0.$ Thus we conclude that the degree of each vertex in X is at most $2$. Since $X$ is a connected graph, $X$ is either a cycle graph or a path graph. From part \ref{Ex:cycle:1} of Example \ref{Ex:cycle} and \ref{ex:path} one can see that the converse also holds.

For the second part of the Proposition, let $\ga(X)$ be a cycle graph on $2k$ $(k\neq 3)$ vertices. From the structure of the adjacency matrix of a cycle graph, we see that when we add the four blocks of $A(C_{2k})$, we obtain $2A(C_{k}).$ On adding all the blocks of $A(\ga(X)),$ we get $2A(L(X))$. We deduce that $L(X)$ is a cycle graph on $k$ vertices. However, we know from \cite{Harary1969} that a connected graph is isomorphic to its line graph if and only if it is a cycle graph. This implies that $X$ is a cycle graph on $k$ vertices, which is a contradiction to Part \ref{Ex:cycle:1} of Example \ref{Ex:cycle}. If $X = K_{1,3}$, then from Part~\ref{Ex:cycle:2} of Example \ref{Ex:cycle} we have already seen that $\ga(X)$ is $C_6$.\\

Proof of Part~\ref{pro:conn:3}. Let $\ga(X)$ have a cut edge and no pendant vertex. From the structure of $A(\ga(X)),$ it can be observed that $\ga(X)$ has two copies of a graph each of whose adjacency matrix is $A_0.$ Since $\ga(X)$ is connected, the edges corresponding to the matrix $B_0$ connects the two copies of the graph given by $A_0.$ As $B_0$ is symmetric, no edge given by the matrix $B_0$ can be a cut edge. Also, note that no edge in the two copies given by $A_0$ in $A(\ga(X))$ can be a cut edge. Therefore $\ga(X)$ has a pendant vertex which implies that $X$ has a pendant vertex that is adjacent to a vertex of degree $2.$
The converse is easy to see as well.\\

Proof of Part \ref{pro:conn:5}. Suppose that $X$ is bipartite with vertex partitions $\{v_1,v_2,\ldots,v_n\}$ and $\{v_1^\prime,v_2^\prime,\ldots,v_k^\prime\}.$ Choose an orientation in such a way that $e_i's$ are the directed edges from $v_i$ to $v_j^\prime$ for all $1\leq i\leq n, 1\leq j \leq k.$ Then observe that
$M=\left[\begin{array}{c|c}
  0 & \B\\
  \hline
  \C & 0
\end{array}
\right]$ which implies
\begin{equation}\label{eqn:bip}
M+M^T=\left[\begin{array}{c|c}
  0 & B_0\\
  \hline
  B_0 & 0
\end{array}
\right].
\end{equation}
Therefore, $\ga(X)$ is bipartite.
The converse is easy to see.
\end{proof}

From the proof of Part \ref{pro:conn:5} of Proposition \ref{pro:conn}, one can see that if $X$ is bipartite, the spectrum of $\ga(X)$ is given by the union of spectra of $A(L(X))$ and $-A(L(X)).$ It is possible to know the number of triangles in $\ga^k(X)$, once we know the number of triangles in $X$ from the following result. 
\begin{proposition}\label{pro:tri}
Let $t_i$ be the number of triangles in $\ga^{i-1}(X),$ where $i\ge 1.$ Then $t_i=2^{i-1}t_1$.
\end{proposition}
\begin{proof}
We shall prove the result by induction on $i$. We begin by proving for $i=2$. It is easy to see that 
$$6t_2= Tr((M+M^T)^{3})=2Tr(M^3)+3Tr(M^2M^T)+3Tr(M(M^T)^2).$$
We now claim that $Tr(M^2M^T)= Tr(M(M^T)^2)=0.$ Since $M$ is a nonnegative matrix, $Tr(M^2M^T)=0$ if and only if $(M^2M^T)_{ii}=0$ for all $i.$ We have $$(M^2M^T)_{ii}=\sum\limits_{k=1}^{2m} (M^2)_{ik}(M^T)_{ki}=\sum\limits_{k=1}^{2m} \sum\limits_{j=1}^{2m}M_{ij}M_{jk}M_{ik}.$$ If each of $M_{ij},M_{jk}$ and $M_{ik}$ are nonzero, then $e_k=e_k^{-1}.$ Consequently, $X$ has multiple edges, which is a contradiction.  Similarly, one can show that $(M(M^2)^T)_{ii}=0.$ Thus, $3t_2=Tr(M^3).$ From Property \ref{property:M:7} of $M$, we have another identity $t_2=\frac{N_3}{3}.$ Hence, the result follows from the fact that  $t_1= \frac{N_3}{6}$, as each vertex  of a triangle can be an initial vertex and two directions.\\
Assume that the result is true for all $i\leq k-1$. Clearly, $t_k=2t_{k-1}$. By the induction hypothesis, the proof is complete.
\end{proof}

 Next, we will present a characterization of symmetric edge graphs analogous to that of line graphs, as given by Krausz in \cite{krausz1943demonstration}. By the {\em star graph at the vertex $u$} in a graph $X,$ denoted by $St(u)$, we mean a subgraph of $X$ with $V(St(u))=\{w\mid w\; \mbox{is adjacent to}\; u\}\cup\{u\}$ and $E(St(u))=\{e\mid u \;\mbox{is incident with}\;e\}.$ The approach used in the proof of Theorem \ref{thm:char} is motivated by the proof of Theorem 8.4 in \cite{Harary1969}.
 
 \begin{theorem}\cite{krausz1943demonstration}\label{thm:char:line}
 A graph is a line graph if and only if its edges can be partitioned into complete subgraphs with the property that no vertex lies in more than two of the subgraphs. 
\end{theorem}

\begin{theorem}\label{thm:char}
A graph is a symmetric edge graph if and only if its edges can be partitioned into crown subgraphs in such a way that each vertex lies in at most two of the subgraphs.
\end{theorem}

\begin{proof}
Let $Y$ be the symmetric edge graph of $X.$ Without loss of generality, $X$ is connected. Let $v$ be any vertex of $X$, then by Part~\ref{ex:path:2} of Example \ref{ex:path} we see that $St(v)$ induces a crown subgraph of $Y$. The edges of $Y$ are exactly in one of the subgraphs. For any $e\in E(X),$ there exists exactly two vertices $a,b\in V(X)$ such that $e\in St(a)\cap St(b),$ which shows that no vertex of $Y$ is in more than two of the subgraphs.

Let $H_1, H_2,\ldots, H_n$ be the partition of the graph $Y$ satisfying the hypothesis. We explain the construction of $X$ from $Y,$ where $Y=\ga(X).$ Let $H=\{H_1,H_2,\ldots,H_n\},$ $U$ be the set of vertices of $Y$ which lies in only one of the partitions $H_i.$ Also, note that $e_i\in U$ if and only if $e_i^{-1}\in U.$ Let $U_1\subset U$ such that $U_1$ contains half of the elements of $U$ and either $e_i$ or $e_i^{-1}\in U_1$. The vertices of $X$ are given by $H\cup U_1$. Two vertices of $X$ are adjacent if they have a nonempty intersection. 
\end{proof}

\begin{corollary}
Let $X$ be a connected graph. Then $\ga(X)$ is unicyclic if and only if $X$ is a tree with $\Delta(X)=3,$ where $\Delta(X)$ denotes the maximum degree of $X$ and there is exactly one vertex of degree three.
\end{corollary}
\begin{proof}
Suppose that $\ga(X)$ is unicyclic, which implies that $X$ does not contain a cycle. By Theorem \ref{thm:char}, it is clear that there does not exist a vertex in $X$ with a degree greater than or equal to $4$.  If there exists more than one vertex of degree $3$, then we get a contradiction to the hypothesis. The converse is easy to follow by Theorem \ref{thm:char}.
\end{proof}

\section{Double covers of line graph}\label{sec:DC}

 Let $X$ be a connected graph with $n$ vertices and $m$ edges. Let $\{v_1,v_2,\ldots,v_n\}$ be the vertex set of $X$. Let $V(X^{\prime \prime})=\{v_1^\prime,v_2^\prime,\ldots,v_n^\prime\}\cup\{v_{n+1}^\prime,v_{n+2}^\prime,\ldots,v_{2n}^\prime\}$ be bipartition of $X^{\prime\prime}$ and $E(X^{\prime \prime})$ be given by
 $$\{e_1,e_2,\ldots,e_m,e_{m+1}=e_1^{-1},e_{m+2}=e_2^{-1},\ldots,e_{2m}=e_m^{-1}\}.$$ We define a map $\phi: V(X^{\prime\prime})\mapsto V(X)$ such that $\phi(v_i^\prime)$ and $\phi(v_{n+i}^\prime)$ are mapped to $v_i$ for all $1\leq i \leq n.$ We label the edges of $X^{\prime \prime}$ such that $e_k$ is an edge from $v_i^\prime$ to $v_{n+j}^\prime$ ($i\neq j$) if and only if $e_{m+k}$ is an edge from $v_j^\prime$ to $v_{n+i}^\prime$ $(i\neq j)$. We illustrate this labelling in Example \ref{Ex:labelling}.
 Recall that the adjacency matrix of a bipartite graph can be written as $\begin{bmatrix}
0 & B\\
B^T & 0
\end{bmatrix}$, where $B$ is called the {\em bi-adjacency matrix}.
 
\begin{example}\label{Ex:labelling} 
In this example, we illustrate the labelling of $X^{\prime\prime}$, where $X$ is given in Figure \ref{fig:3cover}. We label the edges of $X^{\prime\prime}$ in the following manner: 
$$(1^\prime,8^\prime)=e_1,(1^\prime,9^\prime)=e_2,(1^\prime,11^\prime)=e_3,(2^\prime,9^\prime)=e_4,(2^\prime,10^\prime)=e_5,(3^\prime,12^\prime)=e_6,$$
$$(2^\prime,7^\prime)=e_1^{-1},(3^\prime,7^\prime)=e_2^{-1},(5^\prime,7^\prime)=e_3^{-1},(3^\prime,8^\prime)=e_4^{-1},(4^\prime,8^\prime)=e_5^{-1},(6^\prime,9^\prime)=e_6^{-1}.$$

\begin{figure}[ht!]
\begin{subfigure}[ht!]{0.3\textwidth}
\centering
\begin{tikzpicture}[xscale=0.5,yscale=0.7] 
    \vertex (1) at (0,0) [label=left:$6^\prime$]{};
    \vertex (2) at (0,1)[label=left:$5^\prime$]{};
    \vertex (3) at (0,2)[label=left:$4^\prime$]{};
    \vertex (4) at (0,3)[label=left:$3^\prime$]{};
    \vertex (5) at (0,4)[label=left:$2^\prime$]{};
    \vertex (6) at (0,5)[label=left:$1^\prime$]{};
    \vertex (7) at (3,0)[label=right:$12^\prime$]{};
    \vertex (8) at (3,1)[label=right:$11^\prime$] {};
    \vertex (9) at (3,2)[label=right:$10^\prime$]{};
    \vertex (10) at (3,3)[label=right:$9^\prime$]{};
    \vertex (11) at (3,4)[label=right:$8^\prime$]{};
    \vertex (12) at (3,5)[label=right:$7^\prime$]{};
    
\path[-]
(6) edge (11)
(6) edge (10)
(6) edge (8)
(5) edge (12)
(5) edge (10)
(5) edge (9)
(4) edge (12)
(4) edge (11)
(4) edge (7)
(3) edge (11)
(2) edge (12)
(1) edge (10)
;
\end{tikzpicture}
\caption{$X^{\prime\prime}$}
\label{fig:XotimesK2}
\end{subfigure}
\hfill
\begin{subfigure}[ht!]{0.6\textwidth}
         \centering
         \begin{tabular}{|l|llllll|}
\hline
& $7^\prime$ & $8^\prime$ & $9^\prime$ & $10^\prime$ & $11^\prime$ & $12^\prime$\\
\hline
$1^\prime$ & 0 & 1 & 1 & 0 & 1 & 0\\
$2^\prime$ & 1 & 0 & 1 & 1 & 0 & 0\\
$3^\prime$ & 1 & 1 & 0 & 0 & 0 & 1\\
$4^\prime$ & 0 & 1 & 0 & 0 & 0 & 0\\
$5^\prime$ & 1 & 0 & 0 & 0 & 0 & 0\\
$6^\prime$ & 0 & 0 & 1 & 0 & 0 & 0\\
\hline
\end{tabular}\\
\caption{Biadjacency matrix of $X^{\prime\prime}$}
\label{fig:biX}
     \end{subfigure}
\caption{}
\label{fig:exlabel}
\end{figure}
\end{example} 
 
 The rows and columns of $A(L(X^{\prime \prime}))$ are indexed by $E(X^{\prime \prime}).$ It is easy to see that $A(L(X^{\prime\prime}))$ has the following structure $$\begin{bmatrix}
  \mathbb{P} & \mathbb{Q}\\
  \mathbb{Q^T} & \mathbb{R}
  \end{bmatrix},$$ 
where $\mathbb{P,Q,R}$ are $m\times m$ matrices with the following properties:

\begin{enumerate}
    \item $\mathbb{P}=\mathbb{R}$. Since $\mathbb{P}_{ij}=1$ implies that $e_i$ is adjacent to $e_j$, the labelling defined above shows that $e_{m+i}$ is adjacent to $e_{m+j}.$
    
    \item $\mathbb{Q}=\mathbb{Q}^T.$ Since $\mathbb{Q}_{ij}=1$ implies $e_i$ is adjacent to $e_{m+j},$ the labelling defined above shows that $e_j$ is adjacent to $e_{m+i}.$
    
    \item \label{pro:pt:3} $\mathbb{P}+\mathbb{Q}=A(L(X)).$  Note that if $\mathbb{P}_{ij}=1$ then $\mathbb{Q}_{ij}=0$ and vice-versa. If $\mathbb{(P+Q)}_{ij}=\mathbb{P}_{ij}+\mathbb{Q}_{ij}=1,$ then from the definition of covering graph we have $A(L(X))_{ij}=1.$
    \end{enumerate}
     We obtain $L(X^{\prime\prime})$ is a double cover of $L(X)$.
    Also, from the point \ref{pro:pt:3} mentioned above and Theorem \ref{thm:unionspectra} we see that the spectrum of $A(L(X))$ is contained in the spectrum of $A(L(X^{\prime \prime}))$.
    To proceed with the proof of Theorem \ref{thm:Sec:Main}, we need to define claw free graphs. Recall that a claw is another name for the complete bipartite graph $K_{1,3}$. In contrast, a claw-free graph is a graph in which no induced subgraph is a claw. It was proved by Beineke in \cite{beineke1970characterizations} that the line graph of any graph is claw-free.

\begin{proposition}\label{pro:triIhara}
Let $X$ be a connected graph. Then 
\begin{enumerate}

\item \label{pro:triIhara:3}
$L(X^{\prime \prime})$ is disconnected if and only if $X$ is bipartite.
 \item \label{pro:triIhara:1} $2t^\prime=t_2+t_3,$ where $t^\prime,t_2,t_3$ denotes the number of triangles in $L(X),\ga(X)$ and $L(X^{\prime \prime})$, respectively.
\end{enumerate}
\end{proposition}
\begin{proof}
Proof of Part \ref{pro:triIhara:3} is easy to follow from the result proved in \cite{imrich2000s}, which discusses that a Kronecker double cover of a graph $X$ is connected if and only if $X$ is connected and non-bipartite.\\

Proof of Part \ref{pro:triIhara:1}. We know from the definition of a line graph that $t^\prime=t_1+\sum_{i}{d_i\choose 3}$. From Proposition \ref{pro:tri}, we know that $2t_1=t_2.$ Since $X^{\prime \prime}$ is bipartite, we have $t_3=2\sum_{i}{d_i\choose 3}.$ Hence $2t^\prime=t_2+t_3$.
\end{proof}

We are now interested to see  the relationship among $\ga(X),L(X^{\prime\prime})$ and $L(X)^{\prime\prime}$ for a  connected graph $X$. We begin with an example.

\begin{figure}[ht!]
\begin{tabular}{cccc}
\begin{tikzpicture}
    \vertex (1) at (0,0) {};
    \vertex (2) at (1,0) {};
    \vertex (3) at (2,0) {};
    \vertex (4) at (3,0.5) {};
    \vertex (5) at (3,-0.5) {};
    ;
    \path[-]
    (1) edge (2)
    (2) edge (3)
    (3) edge (4)
    (3) edge (5)
    ;
\end{tikzpicture}
&
\begin{tikzpicture}
    \vertex (2) at (1,0) {};
    \vertex (3) at (2,0) {};
    \vertex (4) at (3,0.5) {};
    \vertex (5) at (3,-0.5) {};
    ;
    \path[-]
    (2) edge (3)
    (3) edge (4)
    (3) edge (5)
    (4) edge (5)
    ;
\end{tikzpicture}
&
\begin{tikzpicture}[xscale=0.7]
    \vertex (2) at (1,0) {};
    \vertex (3) at (2,0) {};
    \vertex (4) at (3,0.5) {};
    \vertex (5) at (3,-0.5) {};
    \vertex (6) at (4,-0.5) {};
    \vertex (7) at (4,0.5) {};
    \vertex (8) at (5,0) {};
    \vertex (9) at (6,0) {};
    ;
    \path[-]
    (2) edge (3)
    (3) edge (4)
    (3) edge (5)
    (5) edge (6)
    (7) edge (4)
    (6) edge (8)
    (7) edge (8)
    (8) edge (9)
    ;
\end{tikzpicture}
&
\begin{tikzpicture}[xscale=0.7]
     \vertex (1) at (1,0) {};
    \vertex (2) at (1.5,0.5) {};
    \vertex (3) at (1.5,-0.5) {};
    \vertex (4) at (2.5,0.5) {};
    \vertex (5) at (2.5,-0.5) {};
    \vertex (6) at (3.5,0.5) {};
    \vertex (7) at (3.5,-0.5) {};
    \vertex (8) at (4,0) {};
    ;
    \path[-]
    (1) edge (2)
    (1) edge (3)
    (2) edge (4)
    (3) edge (5)
    (4) edge (6)
    (5) edge (7)
    (6) edge (8)
    (7) edge (8)
    (2) edge (3)
    (6) edge (7)
     ;
\end{tikzpicture}
\end{tabular}\\
    \caption{ $X,L(X),L(X)^{\prime\prime}$ and $\ga(L(X))$ (left to right).}
    \label{fig:Ygraph}
\end{figure}

From Figure \ref{fig:Ygraph}, we see that $\ga(X)=L(X)^{\prime\prime}$ and $\ga(L(X))=L(L(X)^{\prime\prime}),$ but it is not true in general,  one can check with $X=C_3.$ In the next theorem we characterize all those graphs which satisfy this property. 

\begin{table}[ht!]
\scalebox{0.8}{
\begin{tabular}{|c|c|c|}
\hline
     $X$ & $X^{\prime\prime}$ & $\ga(X)=L(X^{\prime\prime})$
     \\
     \hline
      \begin{tikzpicture}   
    \vertex (1) at (0,0) {};
    \vertex (2) at (2,0) {};
    \vertex (3) at (2,2){};
    \vertex (4) at (0,2) {};
     \path[-]
    (1) edge (2)
    (2) edge (3)
    (4) edge (3)
    (1) edge (4)
    (4) edge (2)
    (1) edge (3)
    ;
     \end{tikzpicture}
      &
      \begin{tikzpicture}
    \vertex (2) at (2,1) {};
    \vertex (3) at (3,1) {};
    \vertex (4) at (3,2) {};
    \vertex (5) at (3,0) {};
    \vertex (7) at (5,1) {};
    \vertex (8) at (4,1) {};
    \vertex (9) at (4,2) {};
    \vertex (10) at (4,0){};
    \path[-]
    (2) edge (3)
    (2) edge (4)
    (2) edge (5)
    (4) edge (9)
    (4) edge (8)
    (3) edge (9)
    (3) edge (10)
    (5) edge (8)
    (5) edge (10)
    (7) edge (8)
    (7) edge (9)
    (7) edge (10)
    ;
\end{tikzpicture}
     &
     \begin{tikzpicture}[xscale=0.5,yscale=0.4]
    \vertex (1) at (2,3) {};
    \vertex (2) at (3,1) {};
    \vertex (3) at (1,1) {};
    \vertex (4) at (1,2) {};
    \vertex (5) at (3,2) {};
    \vertex (6) at (2,0) {};
    \vertex (7) at (3.75,3.75){};
    \vertex (8) at (4.5,1.5) {};
    \vertex (9) at (4.5,-1) {};
    \vertex (10) at (0.5,-1){};
    \vertex (12) at (0,3.75) {};
    \vertex (11) at (-1,1.5) {};
    \path[-]
    (1) edge (3)
    (3) edge (2)
    (2) edge (1)
    (6) edge (5)
    (4) edge (5)
    (4) edge (6)
    (1) edge (7)
    (5) edge (7)
    (8) edge (5)
    (2) edge (8)
    (2) edge (9)
    (9) edge (6)
    (3) edge (10)
    (10) edge (6)
    (4) edge (12)
    (1) edge (12)
    (4) edge (11)
    (3) edge (11)
    (7) edge (12)
    (7) edge (8)
    (8) edge (9)
    (9) edge (10)
    (10) edge (11)
    (11) edge (12)
    ;
\end{tikzpicture}\\
\hline
   \begin{tikzpicture}[xscale=0.7,yscale=0.7]   
    \vertex (1) at (0,0) {};
    \vertex (2) at (2,0) {};
    \vertex (3) at (2,2){};
    \vertex (4) at (0,2) {};
    \path[-]
    (1) edge (2)
    (2) edge (3)
    (4) edge (3)
    (1) edge (4)
    (1) edge (3)
    ;
     \end{tikzpicture}  
&
\begin{tikzpicture}   
    \vertex (1) at (0,0) {};
    \vertex (2) at (1,0) {};
    \vertex (3) at (2,0) {};
    \vertex (4) at (3,0) {};
    \vertex (6) at (0,1) {};
    \vertex (7) at (1,1) {};
    \vertex (8) at (2,1) {};
    \vertex (9) at (3,1) {};
    \path[-]
    (1) edge (2)
    (2) edge (3)
    (3) edge (4)
    (6) edge (7)
    (7) edge (8)
    (8) edge (9)
    (1) edge (6)
    (2) edge (7)
    (9) edge (4)
    (3) edge (8)
     ;
     \end{tikzpicture}
 &
  \begin{tikzpicture}   
    \vertex (2) at (1,0) {};
    \vertex (3) at (2,0){};
    \vertex (4) at (3,0) {};
    \vertex (7) at (1,2) {};
    \vertex (8) at (2,2){};
    \vertex (9) at (3,2) {};
    \vertex (11) at (1.5,1.5) {};
    \vertex (12) at (3.5,1.5) {};
    \vertex (13) at (0.5,0.5) {};
    \vertex (14) at (2.5,0.5) {};
    \path[-]
    (2) edge (3)
    (3) edge (4)
    (7) edge (8)
    (8) edge (9)
    (9) edge (12)
    (3) edge (13)
    (3) edge (14)
    (8) edge (11)
    (8) edge (12)
    (2) edge (11)
    (7) edge (13)
    (4) edge (12)
    (2) edge (13)
    (14) edge (9)
    (7) edge (11)
    (4) edge (14)
     ;
     \end{tikzpicture}
\\
\hline
\end{tabular}}
 \caption{}
 \label{table:X,X''}
    \end{table}
\begin{theorem}\label{thm:Sec:Main}
  Let $X$ be a connected graph. Then 
  \begin{enumerate}
  \item \label{thm:Sec:Main:1}
  $\ga(X)$ is isomorphic to $L(X)^{\prime\prime}$ if and only if $X$ is bipartite.
      \item \label{thm:Sec:Main:2} $\ga(X)$ is isomorphic to $L(X^{\prime \prime})$ if and only if one of the following is true: 
      \begin{itemize}
      \item $X$ is a path graph.
      \item $X$ is a cycle graph on even vertices.
      \item  $X=K_4,K_4-\{e\}$ or a triangle with a pendant vertex.
  \end{itemize}
  \item \label{thm:Sec:Main:3} $L(X^{\prime \prime})$ is isomorphic to $L(X)^{\prime\prime}$ if and only if $X$ is either a cycle graph or a path graph.
  \end{enumerate}
  \end{theorem}
 
 \begin{proof}
 Proof of Part \ref{thm:Sec:Main:1}. If $X$ is bipartite, then by Part  \ref{pro:conn:5} of Proposition \ref{pro:conn}, $\ga(X)$ is bipartite which shows that $$A(\ga(X))=A(L(X)^{\prime\prime})=\begin{bmatrix}
0 & A(L(X))\\
A(L(X)) & 0
\end{bmatrix}.$$

Proof of Part \ref{thm:Sec:Main:2}.
 In order to prove this, we first prove that $\ga(X)$ is a line graph of some graph if and only if $|V(X)|\leq 4$ or $X$ is either a cycle graph or a path graph.
 
Suppose that $\ga(X)$ is a line graph of some graph. Clearly $\Delta (X)\leq 3$, since if any vertex $v$ in $X$ has a degree greater than or equal to 4, then by Theorem \ref{thm:char}, $v$ induces a crown graph on at least $8$ vertices. Hence, $\ga(X)$ cannot be a claw-free graph.
\begin{enumerate}
\item[] \textbf{Case 1:} If $\Delta(X)\leq 2,$ then $X$ is either a cycle graph or a path graph. From Part \ref{Ex:cycle:1} of Example \ref{Ex:cycle} and \ref{ex:path}, it is clear that $\ga(X)$ is a line graph of $2X$.

\item[] \textbf{Case 2:} Let $\Delta(X)=3$ and $|V(X)|>4$. Let $v$ be a vertex of degree $3$ and vertices adjacent to $v$ be $x,y,z.$ Since $|V(X)|>4$, if we add a pendant edge on any of the vertices $x,y,z$, then the graph $\ga(X)$ is not a claw-free graph, which is clear from Figure \ref{fig:Ygraph}.
\end{enumerate}

Conversely, if $X=C_n$ (or $P_n$), then $\ga(X)$ is a line graph of two copies of $C_n$ (or $P_n$). If $X=K_{1,3}$, then by Part \ref{Ex:cycle:2} of Example \ref{Ex:cycle} $\ga(X)$ is $C_6$ which is a line graph of $C_6.$
$\ga(X)$ for other non-isomorphic graphs with $|V(X)|=4$ are described in Table \ref{table:X,X''} and Figure \ref{fig:Ygraph}.

Suppose that $\ga(X)=L(X^{\prime \prime}).$ Then by the above statement, it can be noted that $|V(X)|\leq 4$ or $X=C_n$ or $P_n.$ If $X=C_n$ and $n$ is odd, then $L(X^{\prime \prime})=C_{2n}\neq 2C_n=\ga(X).$ If $X=C_n$($n$ is even) or $P_n$, then $L(X^{\prime \prime})=\ga(X).$ For $X=K_4$ or $K_4-\{e\}$ or a triangle with a pendant vertex, we can see from Table \ref{table:X,X''} and Figure \ref{fig:Ygraph} that $\ga(X)=L(X^{\prime \prime})$. If $X=K_{1,3}$ it can be seen that $L(X^{\prime\prime})\neq\ga(X).$ The converse part of the same is easy to follow.\\

Proof of Part \ref{thm:Sec:Main:3}. Assume that $L(X)^{\prime\prime}=L(X^{\prime \prime}).$ From here it is clear that the degree of each vertex of $X$ is less than or equal to two. Hence, $X$ is either a cycle graph or a path graph. Conversely, if $X=C_k$ and $k$ is odd then $X^{\prime \prime}=C_{2k},$ $L(X^{\prime \prime})=L(X)^{\prime \prime}=C_{2k}.$ If $X=C_k$($k$ is even) or $P_k$ then $X^{\prime \prime}=2X$, the result follows.
\end{proof}

We conclude from Theorem \ref{thm:Sec:Main} that if $X(\neq K_4,K_4-e, C_n,$ or a triangle with a pendant vertex) is non-bipartite then $\ga(X)$, $L(X)^{\prime\prime}$ and $L(X^{\prime\prime})$ are three non-isomorphic double covers of $L(X)$.
We have already seen that for a graph $X$, the spectrum of $A(L(X))$ is contained in the spectrum of $A(\ga(X))$ and $A(L(X^{\prime\prime}))$. An immediate question arises about the remaining eigenvalues that is the eigenvalues given by $A_0-B_0$ and $\mathbb{P}-\mathbb{Q}$. If $X$ is bipartite, then $A_0-B_0=-A(L(X))$ and $\mathbb{P}-\mathbb{Q}=A(L(X))$. If $X$ is non-bipartite we have Theorem \ref{thm:eigenvalues}. We shall discuss an example for further clarity.

\begin{example}\label{Ex:eigenvalues}
Let $X$ be the graph given in Figure \ref{fig:3cover}. For a graph $X^{\prime\prime}$ we will continue to use the labelling defined in Example \ref{Ex:labelling}. The adjacency matrix corresponding to $L(X^{\prime\prime})$ is equal to

\medskip
\par
\noindent
$\begin{bmatrix}
\mathbb{P} & \mathbb{Q}\\
\mathbb{Q} & \mathbb{P}
\end{bmatrix},$ where 
$\mathbb{P}=\begin{bmatrix}
0 & 1 & 1 & 0 & 0 & 0\\
1 & 0 & 1 & 1 & 0 & 0\\
1 & 1 & 0 & 0 & 0 & 0\\
0 & 1 & 0 & 0 & 1 & 0\\
0 & 0 & 0 & 1 & 0 & 0\\
0 & 0 & 0 & 0 & 0 & 0
\end{bmatrix}$ 
and
$\mathbb{Q}=\begin{bmatrix}
0 & 0 & 0 & 1 & 1 & 0\\
0 & 0 & 0 & 0 & 0 & 1\\
0 & 0 & 0 & 0 & 0 & 0\\
1 & 0 & 0 & 0 & 0 & 1\\
1 & 0 & 0 & 0 & 0 & 0\\
0 & 1 & 0 & 1 & 0 & 0
\end{bmatrix}. 
$
\par
\noindent
\medskip
It is easy to see that $\mathbb{P}+\mathbb{Q}=A(L(X))$ and\\ $$\mathbb{P}-\mathbb{Q}=\begin{bmatrix}
0 & 1 & 1 & -1 & -1 & 0\\
1 & 0 & 1 & 1 & 0 & -1\\
1 & 1 & 0 & 0 & 0 & 0\\
-1 & 1 & 0 & 0 & 1 & -1\\
-1 & 0 & 0 & 1 & 0 & 0\\
0 & -1 & 0 & -1 & 0 & 0
\end{bmatrix}.$$

\par
\noindent
Now, we use the upper diagonal entries of matrix $\mathbb{P}-\mathbb{Q}$ to assign an orientation to the graph $X$ such that $A_0-B_0=-(\mathbb{P-Q})$. For example: $(\mathbb{P}-\mathbb{Q})_{14}=-1$. From Example \ref{Ex:labelling}, we see that $e_1$ is an edge between $1^\prime$ and $2^\prime+6^\prime$, and $e_4$ is an edge between $2^\prime$ and $3^\prime+6^\prime.$ Hence in $X$, we put $e_1$ from $1$ to $2$ and $e_4$ from $2$ to $3$. Similarly, we repeat the same process for all of the remaining upper diagonal entries in $\mathbb{P-Q}$ and obtained the oriented graph given in Figure \ref{fig:label2}. It is easy to check that for the graph in Figure \ref{fig:label2}, we have $A_0+B_0=A(L(X))$ and $A_0-B_0=-(\mathbb{P}-\mathbb{Q})$.
     
\begin{figure}[ht!]
\tikzset{every picture/.style={line width=0.75pt}} 
\begin{tikzpicture}[x=0.75pt,y=0.75pt,yscale=-1,xscale=1]

\draw  [fill={rgb, 255:red, 0; green, 0; blue, 0 }  ,fill opacity=1 ] (226.78,125.66) .. controls (226.78,123.53) and (228.4,121.8) .. (230.4,121.8) .. controls (232.4,121.8) and (234.03,123.53) .. (234.03,125.66) .. controls (234.03,127.8) and (232.4,129.52) .. (230.4,129.52) .. controls (228.4,129.52) and (226.78,127.8) .. (226.78,125.66) -- cycle ;
\draw    (230.4,125.66) -- (289.11,188.6) ;
\draw    (230.4,125.66) -- (172.42,195.55) ;
\draw    (172.42,195.55) -- (230.4,194.01) -- (293.09,192.85) ;

\draw  [fill={rgb, 255:red, 0; green, 0; blue, 0 }  ,fill opacity=1 ] (171,195.55) .. controls (171,197.68) and (172.59,199.41) .. (174.56,199.41) .. controls (176.52,199.41) and (178.12,197.68) .. (178.12,195.55) .. controls (178.12,193.42) and (176.52,191.69) .. (174.56,191.69) .. controls (172.59,191.69) and (171,193.42) .. (171,195.55) -- cycle ;

\draw  [fill={rgb, 255:red, 0; green, 0; blue, 0 }  ,fill opacity=1 ] (294.16,191.73) .. controls (294.16,193.88) and (292.55,195.63) .. (290.57,195.63) .. controls (288.58,195.63) and (286.97,193.88) .. (286.97,191.73) .. controls (286.97,189.57) and (288.58,187.83) .. (290.57,187.83) .. controls (292.55,187.83) and (294.16,189.57) .. (294.16,191.73) -- cycle ;

\draw    (94.11,194.78) -- (174.56,195.55) ;
\draw  [fill={rgb, 255:red, 0; green, 0; blue, 0 }  ,fill opacity=1 ] (365,188.99) .. controls (365,191.33) and (363.25,193.23) .. (361.08,193.23) .. controls (358.92,193.23) and (357.17,191.33) .. (357.17,188.99) .. controls (357.17,186.64) and (358.92,184.74) .. (361.08,184.74) .. controls (363.25,184.74) and (365,186.64) .. (365,188.99) -- cycle ;
\draw    (290.57,191.73) -- (357.17,188.99) ;
\draw  [fill={rgb, 255:red, 0; green, 0; blue, 0 }  ,fill opacity=1 ] (221.19,68.52) .. controls (221.19,65.75) and (223.26,63.5) .. (225.82,63.5) .. controls (228.37,63.5) and (230.45,65.75) .. (230.45,68.52) .. controls (230.45,71.29) and (228.37,73.54) .. (225.82,73.54) .. controls (223.26,73.54) and (221.19,71.29) .. (221.19,68.52) -- cycle ;
\draw    (225.82,68.52) -- (227.95,95.16) -- (230.4,125.66) ;
\draw   (278.68,169.91) -- (269.54,167.64) -- (272.14,177.42) ;
\draw   (241.91,147) -- (253.24,150.44) -- (249.31,138.41) ;
\draw   (340.44,182.42) -- (348.98,189.37) -- (340.44,196.32) ;
\draw   (317.32,197.97) -- (307.72,191.67) -- (316.95,184.74) ;
\draw   (261.41,200.18) -- (253.58,193.23) -- (261.41,186.28) ;
\draw   (204.49,187.04) -- (216.56,194.8) -- (204.46,202.5) ;
\draw   (186,172.78) -- (194.9,168.96) -- (195.2,179.34) ;
\draw   (219,146.83) -- (209.79,150.63) -- (209.44,139.95) ;
\draw   (148.22,189.37) -- (157.47,195.36) -- (148.22,201.34) ;
\draw   (233.35,113.68) -- (228.43,101.54) -- (224.46,114.08) ;
\draw   (222.59,83.35) -- (228.17,90.71) -- (232.89,82.67) ;

\draw  [fill={rgb, 255:red, 0; green, 0; blue, 0 }  ,fill opacity=1 ] (88,193.88) .. controls (88,197.45) and (90.74,200.34) .. (94.11,200.34) .. controls (97.48,200.34) and (100.22,197.45) .. (100.22,193.88) .. controls (100.22,190.31) and (97.48,187.41) .. (94.11,187.41) .. controls (90.74,187.41) and (88,190.31) .. (88,193.88) -- cycle ;
\draw   (132.02,201.14) -- (117.44,194.66) -- (131.92,187.92) ;

\draw (193,123) node [anchor=north west][inner sep=0.75pt]   [align=left] {$e_2$};
\draw (170,152) node [anchor=north west][inner sep=0.75pt]   [align=left] {$e_2^{-1}$};
\draw (118,169) node [anchor=north west][inner sep=0.75pt]   [align=left] {$e_6$};
\draw (148,202) node [anchor=north west][inner sep=0.75pt]   [align=left] {$e_6^{-1}$};
\draw (333,162) node [anchor=north west][inner sep=0.75pt]   [align=left] {$e_5$};
\draw (302,201) node [anchor=north west][inner sep=0.75pt]   [align=left] {$e_5^{-1}$};
\draw (253,127) node [anchor=north west][inner sep=0.75pt]   [align=left] {$e_1$};
\draw (278,150) node [anchor=north west][inner sep=0.75pt]   [align=left] {$e_1^{-1}$};
\draw (197,204) node [anchor=north west][inner sep=0.75pt]   [align=left] {$e_4^{-1}$};
\draw (250,204) node [anchor=north west][inner sep=0.75pt]   [align=left] {$e_4$};
\draw (237,101) node [anchor=north west][inner sep=0.75pt]   [align=left] {$e_3$};
\draw (238.45,77.94) node [anchor=north west][inner sep=0.75pt]   [align=left] {$e_3^{-1}$};
\draw (225.78,128.66) node [anchor=north west][inner sep=0.75pt]   [align=left] {1};
\draw (271,177) node [anchor=north west][inner sep=0.75pt]   [align=left] {2};
\draw (183,181) node [anchor=north west][inner sep=0.75pt]   [align=left] {3};
\draw (370,179) node [anchor=north west][inner sep=0.75pt]   [align=left] {4};
\draw (206,59) node [anchor=north west][inner sep=0.75pt]   [align=left] {5};
\draw (87,164) node [anchor=north west][inner sep=0.75pt]   [align=left] {6};
\end{tikzpicture}
\centering
\caption{}
\label{fig:label2}
\end{figure}
\end{example}

Now using the idea of Example \ref{Ex:eigenvalues}, we prove that $A_0-B_0=-(\mathbb{P}-\mathbb{Q})$.
\begin{theorem}\label{thm:eigenvalues}
Let $X$ be a connected graph. Then $A_0-B_0=-(\mathbb{P}-\mathbb{Q})$.
\end{theorem}
\begin{proof}
If $X$ is bipartite, then we are done. Suppose that $X$ is non-bipartite. It is clear that $A_0-B_0$ and $\mathbb{P}-\mathbb{Q}$ have zero entries at the same positions. Suppose that $(\mathbb{P}-\mathbb{Q})_{ij}=-1$. This implies that edge $e_i$ is adjacent to $e_{m+j}$ in $X^{\prime\prime}$. Let $e_i$ be an edge between vertices $v_a^{\prime}$ and $v_{n+b}^{\prime}$, and $e_{m+j}$ be an edge between vertices $v_a^{\prime}$ and $v_{n+c}^\prime$. In graph $X$, we label the edge from vertex $v_b$ to $v_a$ as $e_i$ and the edge from $v_a$ to $v_c$ as $e_j$. This shows that $(A_0-B_0)_{ij}=1$. 
\end{proof}
Recall that two graphs of the same order are called {\em equienergetic (resp., cospectral)} if they have the same energy (resp., spectrum). In \cite{balakrishnan2004energy} Balakrishnan showed that for any integer $k\geq 3$, there exist two equienergetic graphs of order $4k$ that are not cospectral. Let $X$ be a graph on $m$ edges where $m\geq5$. Then from Theorem \ref{thm:eigenvalues}, we see that $\ga(X)$ and $L(X^{\prime\prime})$ are equienergetic graphs of order $2m$ that are not cospectral. We exclude the graphs given in Part \ref{thm:Sec:Main:2} of Theorem \ref{thm:Sec:Main}.
Using Theorem \ref{thm:eigenvalues}, we will provide a relation between the zeta function of $\ga(X)$ and $L(X^{\prime\prime})$ in Corollary \ref{cor:zeta}. Consequently, we obtain that the zeta function of $L(X)$ divides the zeta function of $\ga(X),L(X^{\prime\prime})$ and $L(X)^{\prime\prime}$. The Kronecker product of  matrices $A=[a_{ij}]$ and $B$ is defined to be the partitioned matrix $[a_{ij}B]$ and is denoted by $A\otimes B$.

\begin{corollary}\label{cor:zeta}
Let $X$ be a connected graph with $m$ edges. Then 
\begin{align*}
 \ze_{\ga(X)}^{-1}(u)&=\ze_{L(X)}^{-1}(u)g(u),\\
 \ze_{L(X^{\prime\prime})}^{-1}(u)&=\ze_{L(X)}^{-1}(u)g(-u),\\
 \ze_{L(X)^{\prime\prime}}^{-1}(u)&=\ze_{L(X)}^{-1}(u)\ze_{L(X)}^{-1}(-u),
\end{align*}
where $g(u)=(1-u^2)^{|E(L(X))|-|V(L(X))|}det(\I_{m}-(A_0-B_0)u+Q(L(X))u^2)$.
\end{corollary}
\begin{proof}
Let $P=\begin{bmatrix}
\I_m & \I_m\\
\I_m & -\I_m
\end{bmatrix}$. Then $PA(\ga(X))P^{-1}=\begin{bmatrix}
C_0 & 0\\
0 & D_0
\end{bmatrix}$, where $A_0+B_0=C_0$ and $D_0=A_0-B_0$ and $PQ(\ga(X))P^{-1}=Q(L(X))\otimes \I_2$. Let $s=|E(\ga(X))|-|V(\ga(X))|.$
From Equation \ref{eqn:1} we have,
\begin{align*}
\ze_{\ga(X)}^{-1}(u)&=(1-u^2)^{s} det(\I_{2m}-A(\ga(X))u+Q(\ga(X))u^2)\\
&=(1-u^2)^{s} det(P(\I_{2m}-A(\ga(X))u+Q(\ga(X))u^2)P^{-1})\\
&=(1-u^2)^{s} det(\I_{2m}-PA(\ga(X))P^{-1}u+Q(\ga(X))u^2)\\
&=(1-u^2)^{s} det(\I_{m}-C_0u+Q(L(X))u^2)det(I_{m}-D_0u+Q(L(X))u^2)\\
&=\ze_{L(X)}^{-1}(u)g(u).
\end{align*}

Similarly, we can see that 
$$\ze_{L(X^{\prime\prime})}^{-1}(u)=\ze_{L(X)}^{-1}(u)(1-u^2)^{|E(L(X))|-|V(L(X))|}det(\I_{m}-(\mathbb{P}-\mathbb{Q})u+Q(L(X))u^2).$$

\par
\noindent
By Theorem \ref{thm:eigenvalues}, we obtain $\ze_{L(X^{\prime\prime})}^{-1}(u)=\ze_{L(X)}^{-1}(u)g(-u)$. As
$A(L(X)^{\prime\prime})=\begin{bmatrix}
0 & A(L(X))\\
A(L(X)) & 0
\end{bmatrix},$ we use the above technique which provides \linebreak $\ze_{L(X)^{\prime\prime}}^{-1}(u)=\ze_{L(X)}^{-1}(u)\ze_{L(X)}^{-1}(-u)$.
This completes the proof.
\end{proof}
From the above corollary, we conclude that if $X$ is bipartite, then 
$\ze_{\ga(X)}^{-1}(u)= \ze_{L(X)}^{-1}(u)\ze_{L(X)}^{-1}(-u)$ and $\ze_{L(X^{\prime \prime})}^{-1}(u)= (\ze_{L(X)}^{-1}(u))^2.$

\bibliographystyle{amsplain}
\bibliography{bibliography}

\providecommand{\bysame}{\leavevmode\hbox to3em{\hrulefill}\thinspace}
\providecommand{\MR}{\relax\ifhmode\unskip\space\fi MR }
\providecommand{\MRhref}[2]{%
  \href{http://www.ams.org/mathscinet-getitem?mr=#1}{#2}
}
\providecommand{\href}[2]{#2}
\begin{thebibliography}{10}

\bibitem{balakrishnan2004energy}
R~Balakrishnan, \emph{The energy of a graph}, Linear Algebra and its
  Applications \textbf{387} (2004), 287--295.

\bibitem{bass1992ihara}
Hyman Bass, \emph{The ihara-selberg zeta function of a tree lattice},
  International Journal of Mathematics \textbf{3} (1992), no.~06, 717--797.

\bibitem{beineke1970characterizations}
Lowell~W Beineke, \emph{Characterizations of derived graphs}, Journal of
  Combinatorial theory \textbf{9} (1970), no.~2, 129--135.

\bibitem{davis2013circulant}
Philip~J Davis, \emph{Circulant matrices}, American Mathematical Soc., 2013.

\bibitem{gross1977generating}
Jonathan~L Gross and Thomas~W Tucker, \emph{Generating all graph coverings by
  permutation voltage assignments}, Discrete Mathematics \textbf{18} (1977),
  no.~3, 273--283.

\bibitem{Harary1969}
F.~Harary, \emph{Graph theory}, Addison-Wesley, 1969.

\bibitem{MR1040609}
Ki-ichiro Hashimoto, \emph{Zeta functions of finite graphs and representations
  of {$p$}-adic groups}, Automorphic forms and geometry of arithmetic
  varieties, Adv. Stud. Pure Math., vol.~15, Academic Press, Boston, MA, 1989,
  pp.~211--280.

\bibitem{horton2006ihara}
Matthew~D Horton, \emph{Ihara zeta functions of irregular graphs}, Ph.D.
  thesis, UC San Diego, 2006.

\bibitem{imrich2000s}
W~Imrich, \emph{S. klavžar, product graphs}, 2000.

\bibitem{krausz1943demonstration}
J{\'o}zsef Krausz, \emph{D{\'e}monstration nouvelle d’une th{\'e}oreme de
  whitney sur les r{\'e}seaux}, Mat. Fiz. Lapok \textbf{50} (1943), no.~1,
  75--85.

\bibitem{minei2018ramanujan}
Marvin Minei and Howard Skogman, \emph{Ramanujan graphs arising as weighted
  galois covering graphs}, Electronic Journal of Graph Theory and Applications
  \textbf{6} (2018), no.~1, 237869.

\bibitem{biggs1993algebraic}
N.L.Biggs, \emph{Algebraic graph theory}, vol.~67, Cambridge university press,
  1993.

\bibitem{ray1967characterization}
DK~Ray-Chaudhuri, \emph{Characterization of line graphs}, Journal of
  Combinatorial Theory \textbf{3} (1967), no.~3, 201--214.

\bibitem{terras2010zeta}
Audrey Terras, \emph{Zeta functions of graphs: a stroll through the garden},
  vol. 128, Cambridge University Press, 2010.

\bibitem{whitney1992congruent}
Hassler Whitney, \emph{Congruent graphs and the connectivity of graphs},
  Hassler Whitney Collected Papers, Springer, 1992, pp.~61--79.

\end{thebibliography}

\begin{table}[ht!]
    \centering
    \scalebox{0.7}{
    \begin{tabular}{|c|c||c|c|} 
    \hline
    $X$ & $\ga(X)$ & $X$ & $\ga(X)$\\
    \hline
    \begin{tikzpicture}    
    \vertex (1) at (1,1) {};
    \vertex (2) at (2,1) {};
    \vertex (3) at (3,2) {};
    \vertex (4) at (3,0) {};
    \vertex (5) at (4,2) {};
    \vertex (6) at (4,0) {};
    \path[-]
    (1) edge (2)
    (2) edge (3)
    (2) edge (4)
    (3) edge (5)
    (4) edge (6)
     ;
     \end{tikzpicture}
     &
    \begin{tikzpicture}[xscale=0.6,yscale=0.6]
    \vertex (1) at (1,2) {};
    \vertex (2) at (2,1) {};
    \vertex (3) at (3,2) {};
    \vertex (4) at (4,1) {};
    \vertex (5) at (5,2) {};
    \vertex (6) at (1,-1){};
    \vertex (7) at (2,0) {};
    \vertex (8) at (3,-1){};
    \vertex (9) at (4,0) {};
    \vertex (10) at (5,-1){};
     \path[-]
    (1) edge (2)
    (2) edge (3)
    (3) edge (4)
    (4) edge (5)
    (6) edge (7)
    (7) edge (8)
    (8) edge (9)
    (9) edge (10)
    (2) edge (7)
    (4) edge (9)
;
\end{tikzpicture}
&
\begin{tikzpicture}   
    \vertex (1) at (1,1) {};
    \vertex (2) at (2,1) {};
    \vertex (3) at (1.5,2){};
    \vertex (4) at (1,0) {};
    \vertex (5) at (2,0) {};
    \path[-]
    (1) edge (2)
    (2) edge (3)
    (1) edge (3)
    (1) edge (4)
    (2) edge (5)
    (4) edge (5)
     ;
     \end{tikzpicture}
&
\begin{tikzpicture}[yscale=0.5]
    \vertex (1) at (1,0) {};
    \vertex (2) at (2,0) {};
    \vertex (3) at (3,0) {};
    \vertex (4) at (4,0) {};
    \vertex (5) at (1.5,-1) {};
    \vertex (6) at (3.5,-1) {};
    \vertex (7) at (3,-3) {};
    \vertex (8) at (2,-2) {};
    \vertex (9) at (1.5,-4) {};
    \vertex (10) at (3.5,-4){};
    \vertex (11) at (4,-2) {};
    \vertex (12) at (5,-4) {};
    \path[-]
    (1) edge (2)
    (2) edge (3)
    (3) edge (4)
    (3) edge (5)
    (3) edge (6)
    (8) edge (9)
    (7) edge (8)
    (7) edge (9)
    (7) edge (10)
    (10) edge (12)
    (11) edge (12)
    (7) edge (11)
    (4) edge (11)
    (6) edge (10)
    (2) edge (8)
    (9) edge (5)
    (4) edge (6)
    (1) edge (5)
    ;
\end{tikzpicture}
    \\
\hline
\begin{tikzpicture}
    \vertex (1) at (1,1) {};
    \vertex (2) at (2,1) {};
    \vertex (3) at (3,1) {};
    \vertex (4) at (4,2) {};
    \vertex (5) at (4,0) {};
    \vertex (6) at (4,1) {};
    \path[-]
    (1) edge (2)
    (2) edge (3)
    (3) edge (4)
    (3) edge (5)
    (3) edge (6)
     ;
\end{tikzpicture}&
\begin{tikzpicture}
    \vertex (1) at (1,1) {};
    \vertex (2) at (2,1) {};
    \vertex (3) at (3,1) {};
    \vertex (4) at (3,2) {};
    \vertex (5) at (3,0) {};
    \vertex (6) at (6,1) {};
    \vertex (7) at (5,1) {};
    \vertex (8) at (4,1) {};
    \vertex (9) at (4,2) {};
    \vertex (10) at (4,0){};
    \path[-]
    (1) edge (2)
    (2) edge (3)
    (2) edge (4)
    (2) edge (5)
    (4) edge (9)
    (4) edge (8)
    (3) edge (9)
    (3) edge (10)
    (5) edge (8)
    (5) edge (10)
    (7) edge (8)
    (7) edge (9)
    (7) edge (10)
    (7) edge (6)
    ;
\end{tikzpicture}
&
\begin{tikzpicture}   
    \vertex (1) at (0,0) {};
    \vertex (2) at (1,0) {};
    \vertex (3) at (1,1){};
    \vertex (4) at (0,1) {};
    \vertex (5) at (1.5,1.5) {};
    \path[-]
    (1) edge (2)
    (2) edge (3)
    (4) edge (3)
    (1) edge (4)
    (3) edge (5)
     ;
     \end{tikzpicture}
&
\begin{tikzpicture}   
    \vertex (1) at (0,0) {};
    \vertex (2) at (1,0) {};
    \vertex (3) at (2,0){};
    \vertex (4) at (3,0) {};
    \vertex (5) at (4,0) {};
    \vertex (6) at (0,1) {};
    \vertex (7) at (1,1) {};
    \vertex (8) at (2,1){};
    \vertex (9) at (3,1) {};
    \vertex (10) at (4,1) {};
    \path[-]
    (1) edge (2)
    (2) edge (3)
    (3) edge (4)
    (5) edge (4)
    (6) edge (7)
    (7) edge (8)
    (8) edge (9)
    (9) edge (10)
    (1) edge (6)
    (2) edge (7)
    (10) edge (5)
    (9) edge (4)
     ;
     \end{tikzpicture}
\\
\hline

\begin{tikzpicture}
    \vertex (1) at (1,1) {};
    \vertex (2) at (2,1) {};
    \vertex (3) at (3,1) {};
    \vertex (4) at (4,2) {};
    \vertex (5) at (4,0) {};
    \vertex (6) at (0,2) {};
    \vertex (7) at (0,0) {};
    \path[-]
    (1) edge (2)
    (2) edge (3)
    (3) edge (4)
    (3) edge (5)
    (1) edge (6)
    (1) edge (7)
    ;
\end{tikzpicture}
&
\begin{tikzpicture}[xscale=0.6,yscale=0.6]
    \vertex (1) at (0,0) {};
    \vertex (2) at (1,1) {};
    \vertex (3) at (1,0) {};
    \vertex (4) at (2,1) {};
    \vertex (5) at (2,0) {};
    \vertex (6) at (3,0) {};
    \vertex (7) at (0,-1) {};
    \vertex (8) at (1,-2) {};
    \vertex (9) at (1,-1) {};
    \vertex (10) at (2,-2) {};
    \vertex (11) at (2,-1) {};
    \vertex (12) at (3,-1) {};
    \path[-]
    (1) edge (2)
    (2) edge (3)
    (2) edge (4)
    (4) edge (5)
    (4) edge (6)
    (7) edge (8)
    (8) edge (9)
    (8) edge (10)
    (10) edge (11)
    (10) edge (12)
    (1) edge (9)
    (3) edge (7)
    (5) edge (12)
    (6) edge (11)
     ;
\end{tikzpicture}
&
     \begin{tikzpicture}   
    \vertex (1) at (0,0) {};
    \vertex (2) at (2,0) {};
    \vertex (3) at (2,1) {};
    \vertex (4) at (2,2) {};
    \vertex (5) at (0,2) {};
    \vertex (6) at (0,1) {};
    \path[-]
    (1) edge (2)
    (2) edge (3)
    (4) edge (3)
    (5) edge (4)
    (6) edge (5)
    (3) edge (6)
    (1) edge (6)
     ;
     \end{tikzpicture}
     &
     \begin{tikzpicture} [yscale=0.8]  
    \vertex (1) at (0,0) {};
    \vertex (2) at (1,0) {};
    \vertex (3) at (2,0){};
    \vertex (4) at (3,0) {};
    \vertex (5) at (4,0) {};
    \vertex (6) at (0,2) {};
    \vertex (7) at (1,2) {};
    \vertex (8) at (2,2){};
    \vertex (9) at (3,2) {};
    \vertex (10) at (4,2) {};
    \vertex (11) at (1.5,1.5) {};
    \vertex (12) at (3.5,1.5) {};
    \vertex (13) at (0.5,0.5) {};
    \vertex (14) at (2.5,0.5) {};
    \path[-]
    (1) edge (2)
    (2) edge (3)
    (3) edge (4)
    (5) edge (4)
    (6) edge (7)
    (7) edge (8)
    (8) edge (9)
    (9) edge (10)
    (3) edge (13)
    (3) edge (14)
    (8) edge (11)
    (8) edge (12)
    (6) edge (11)
    (10)edge (12)
    (1) edge (13)
    (5) edge (14)
    (2) edge (11)
    (7) edge (13)
    (9) edge (14)
    (4) edge (12)
     ;
     \end{tikzpicture}
    \\
\hline
\begin{tikzpicture}   
    \vertex (1) at (1,1) {};
    \vertex (2) at (2,1) {};
    \vertex (3) at (1.5,2) {};
    \vertex (4) at (1,3) {};
    \vertex (5) at (1.5,3) {};
    \path[-]
    (1) edge (2)
    (2) edge (3)
    (1) edge (3)
    (3) edge (4)
    (3) edge (5)
     ;
     \end{tikzpicture}
&
\begin{tikzpicture}
    \vertex (1) at (1,1) {};
    \vertex (2) at (2,1) {};
    \vertex (3) at (3,1) {};
    \vertex (4) at (3,2) {};
    \vertex (5) at (3,0) {};
    \vertex (6) at (6,1) {};
    \vertex (7) at (5,1) {};
    \vertex (8) at (4,1) {};
    \vertex (9) at (4,2) {};
    \vertex (10) at(4,0){};
    \path[-]
    (1) edge (2)
    (2) edge (3)
    (2) edge (4)
    (2) edge (5)
    (4) edge (9)
    (4) edge (8)
    (3) edge (9)
    (3) edge (10)
    (5) edge (8)
    (5) edge (10)
    (7) edge (8)
    (7) edge (9)
    (7) edge (10)
    (7) edge (6)
    (1) edge (5)
    (6) edge (9)
    ;
\end{tikzpicture}
 &
      \begin{tikzpicture}[yscale=0.7]
    \vertex (1) at (2,3) {};
    \vertex (2) at (2,2) {};
    \vertex (3) at (3,2) {};
    \vertex (4) at (4,2) {};
    \vertex (5) at (5,2) {};
    \vertex (6) at (2.5,1) {};
    \vertex (7) at (3.2,1){};
    \vertex (8) at (3.7,1) {};
    \path[-]
    (1) edge (2)
    (2) edge (3)
    (3) edge (4)
    (4) edge (5)
    (3) edge (6)
    (3) edge (7)
    (4) edge (8)
    ;
\end{tikzpicture}
&
\begin{tikzpicture}[yscale=0.6]
    \vertex (1) at (1,2) {};
    \vertex (2) at (2,2) {};
    \vertex (3) at (4,2) {};
    \vertex (4) at (1.5,1) {};
    \vertex (5) at (2.5,1) {};
    \vertex (7) at (4,1) {};
    \vertex (8) at (5,1) {};
    \vertex (1') at (1,-2) {};
    \vertex (2') at (2,-2) {};
    \vertex (3') at (4,-2) {};
    \vertex (4') at (1.5,-1) {};
    \vertex (5') at (2.5,-1) {};
    \vertex (7') at (4,-1) {};
    \vertex (8') at (5,-1) {};
    \path[-]
    (1) edge (2)
    (2) edge (3)
    (2) edge (4)
    (2) edge (5)
    (3) edge (7)
    (3) edge (8)
    (1') edge (2')
    (2') edge (3')
    (2') edge (4')
    (2') edge (5')
    (3') edge (7')
    (3') edge (8')
    (4) edge (5')
    (4) edge (3')
    (5) edge (4')
    (5) edge (3')
    (3) edge (4')
    (3) edge (5')
    (7) edge (8')
    (8) edge (7')
    ;
\end{tikzpicture}
\\
\hline
\begin{tikzpicture}   
    \vertex (1) at (1,1) {};
    \vertex (2) at (2,1) {};
    \vertex (3) at (1.5,2) {};
    \vertex (6) at (0,1) {};
     \vertex (5) at (3,1) {};
     \path[-]
    (1) edge (2)
    (2) edge (3)
    (1) edge (3)
    (1) edge (6)
    (2) edge (5)
    ;
     \end{tikzpicture}
&
\begin{tikzpicture}   
    \vertex (2) at (1,0) {};
    \vertex (3) at (2,0){};
    \vertex (4) at (3,0) {};
    \vertex (7) at (1,2) {};
    \vertex (8) at (2,2){};
    \vertex (9) at (3,2) {};
    \vertex (11) at (1.5,1.5) {};
    \vertex (12) at (3.5,1.5) {};
    \vertex (13) at (0.5,0.5) {};
    \vertex (14) at (2.5,0.5) {};
    \path[-]
    (2) edge (3)
    (3) edge (4)
    (7) edge (8)
    (8) edge (9)
    (9) edge (12)
    (3) edge (13)
    (3) edge (14)
    (8) edge (11)
    (8) edge (12)
    (2) edge (11)
    (7) edge (13)
    (4) edge (12)
    (2) edge (13)
    (14) edge (9)
     ;
     \end{tikzpicture}
&
\begin{tikzpicture}
    \vertex (1) at (1,0) {};
    \vertex (2) at (2,0) {};
    \vertex (3) at (3,0) {};
    \vertex (4) at (4,0) {};
    \vertex (5) at (1.5,1) {};
    \vertex (6) at (3.5,1) {};
    \vertex (7) at (2.5,0) {};
    \vertex (8) at (2.5,1) {};
    \vertex (9) at (1.5,1.5) {};
    \vertex (10) at (3.5,1.5) {};
    \path[-]
    (1) edge (2)
    (3) edge (4)
    (1) edge (5)
    (2) edge (5)
    (4) edge (6)
    (3) edge (6)
    (7) edge (8)
    (8) edge (9)
    (8) edge (10)
    (9) edge (10)
    (2) edge (7)
    (7) edge (3)
    ;
\end{tikzpicture}
&
 \begin{tikzpicture}[xscale=0.5,yscale=0.5]
    \vertex (1) at (-2,2) {};
    \vertex (2) at (-1,2) {};
    \vertex (3) at (1,2) {};
    \vertex (4) at (2,2) {};
    \vertex (5) at (-2,-2) {};
    \vertex (6) at (-1,-2) {};
    \vertex (7) at (1,-2) {};
    \vertex (8) at (2,-2){};
    \vertex (9) at (0,1){};
    \vertex (10) at (0,0) {};
    \vertex (11) at (-1.5,3){};
    \vertex (12) at (1.5,3) {};
    \vertex (13) at (-1.5,-3){};
    \vertex (14) at (1.5,-3){};
   \vertex (15) at (-0.5,-1){};
    \vertex (16) at (-1,-0.5){};
    \vertex (17) at (0.5,-1){};
    \vertex (18) at (1,-0.5){};
     \vertex (19) at (-2,-0.5){};
    \vertex (20) at (-2,-1){};
    \vertex (21) at (-3,-0.75){};
    \vertex (22) at (2,-1){};
    \vertex (23) at (2,-0.5){};
    \vertex (24) at (3,-0.75){};
    \path[-]
    (1) edge (2)
    (2) edge (9)
    (3) edge (9)
    (3) edge (4)
    (4) edge (10)
    (1) edge (10)
    (10) edge (17)
    (9) edge (16)
    (5) edge (6)
    (7) edge (8)
    (9) edge (18)
   (10) edge (15)
   (1) edge (11)
   (2) edge (11)
   (3) edge (12)
   (4) edge (12)
   (5) edge (13)
   (6) edge (13)
   (7) edge (14)
   (8) edge (14)
   (16) edge (5)
   (6) edge (15)
   (17) edge (7)
   (18) edge (8)
   (16) edge (17)
   (15) edge (18)
   (16) edge (19)
   (15) edge (20)
   (19) edge (20)
   (19) edge (21)
   (21) edge (20)
   (17) edge (22)
   (18) edge (23)
   (22) edge (23)
   (23) edge (24)
   (22) edge (24)
   
   ;
\end{tikzpicture}
\\

\hline
\begin{tikzpicture}   
    \vertex (1) at (1,1) {};
    \vertex (2) at (2,1) {};
    \vertex (3) at (1.5,2) {};
    \vertex (6) at (0,1) {};
    \vertex (4) at (1.5,3) {};
    \vertex (5) at (3,1) {};
     \path[-]
    (1) edge (2)
    (2) edge (3)
    (1) edge (3)
    (1) edge (6)
    (2) edge (5)
    (3) edge (4)
     ;
     \end{tikzpicture}
&
\begin{tikzpicture}[xscale=0.5,yscale=0.5]
    \vertex (1) at (2,3) {};
    \vertex (2) at (3,1) {};
    \vertex (3) at (1,1) {};
    \vertex (4) at (1,2) {};
    \vertex (5) at (3,2) {};
    \vertex (6) at (2,0) {};
    \vertex (7) at (3.75,3.75) {};
    \vertex (8) at (4.5,1.5) {};
    \vertex (9) at (4.5,-1) {};
    \vertex (10) at (0.5,-1){};
    \vertex (12) at (0,3.75) {};
    \vertex (11) at (-1,1.5) {};
    \path[-]
    (1) edge (3)
    (3) edge (2)
    (2) edge (1)
    (6) edge (5)
    (4) edge (5)
    (4) edge (6)
    (1) edge (7)
    (5) edge (7)
    (8) edge (5)
    (2) edge (8)
    (2) edge (9)
    (9) edge (6)
    (3) edge (10)
    (10) edge (6)
    (4) edge (12)
    (1) edge (12)
    (4) edge (11)
    (3) edge (11)
    ;
\end{tikzpicture}
&
\begin{tikzpicture}
    \vertex (1) at (1,0) {};
    \vertex (2) at (2,0) {};
    \vertex (3) at (3,0) {};
    \vertex (4) at (4,0) {};
    \vertex (5) at (1.5,1) {};
    \vertex (6) at (3.5,1) {};
    \path[-]
    (1) edge (2)
    (2) edge (3)
    (3) edge (4)
    (1) edge (5)
    (2) edge (5)
    (4) edge (6)
    (3) edge (6)
    ;
\end{tikzpicture}
&
    \begin{tikzpicture}[xscale=0.5,yscale=0.5]
    \vertex (1) at (-2,2) {};
    \vertex (2) at (-1,2) {};
    \vertex (3) at (1,2) {};
    \vertex (4) at (2,2) {};
    \vertex (5) at (-2,-2) {};
    \vertex (6) at (-1,-2) {};
    \vertex (7) at (1,-2) {};
    \vertex (8) at (2,-2) {};
    \vertex (9) at (0,1){};
    \vertex (10) at (0,0) {};
    \vertex (11) at (-1.5,3) {};
    \vertex (12) at (1.5,3) {};
    \vertex (13) at (-1.5,-3) {};
    \vertex (14) at (1.5,-3) {};
   \path[-]
    (1) edge (2)
    (2) edge (9)
    (3) edge (9)
    (3) edge (4)
    (4) edge (10)
    (1) edge (10)
    (10) edge (7)
    (9) edge (5)
    (5) edge (6)
    (7) edge (8)
    (9) edge (8)
   (10) edge (6)
   (1) edge (11)
   (2) edge (11)
   (3) edge (12)
   (4) edge (12)
   (5) edge (13)
   (6) edge (13)
   (7) edge (14)
   (8) edge (14)
    ;
\end{tikzpicture}
\\
\hline
\begin{tikzpicture}   
    \vertex (1) at (1,1) {};
    \vertex (2) at (2,1) {};
    \vertex (3) at (1.5,2) {};
    \vertex (4) at (2,3) {};
    \vertex (5) at (1,3) {};
     \path[-]
    (1) edge (2)
    (2) edge (3)
    (1) edge (3)
    (3) edge (4)
    (3) edge (5)
    (5) edge (4)
     ;
     \end{tikzpicture}
&
\begin{tikzpicture}
    \vertex (1) at (1,2) {};
    \vertex (2) at (2,1) {};
    \vertex (3) at (3,1) {};
    \vertex (4) at (3,2) {};
    \vertex (5) at (3,0) {};
    \vertex (6) at (6,0) {};
    \vertex (7) at (5,1) {};
    \vertex (8) at (4,1) {};
    \vertex (9) at (4,2) {};
    \vertex (10) at (4,0){};
    \vertex (11) at (3.5,3) {};
    \vertex (12) at (3.5,-1) {};
    \path[-]
    (1) edge (2)
    (2) edge (3)
    (2) edge (4)
    (2) edge (5)
    (4) edge (9)
    (4) edge (8)
    (3) edge (9)
    (3) edge (10)
    (5) edge (8)
    (5) edge (10)
    (7) edge (8)
    (7) edge (9)
    (7) edge (10)
    (7) edge (6)
    (1) edge (3)
    (6) edge (8)
    (11) edge (4)
    (9) edge (11)
    (12) edge (5)
    (12) edge (10)
    ;
\end{tikzpicture}
&
\begin{tikzpicture}   
    \vertex (1) at (1,1) {};
    \vertex (2) at (2,1) {};
    \vertex (3) at (1.5,2) {};
    \vertex (4) at (2.5,1) {};
    \vertex (5) at (3.5,1) {};
    \vertex (6) at (2.5,2) {};
    \vertex (7) at (3.5,2) {};
    \path[-]
    (1) edge (2)
    (2) edge (3)
    (1) edge (3)
    (2) edge (4)
    (4) edge (5)
    (5) edge (7)
    (6) edge (7)
    (4) edge (6)
     ;
     \end{tikzpicture}
     &
     \begin{tikzpicture}[xscale=0.5,yscale=0.5]
    \vertex (1) at (-2,2) {};
    \vertex (2) at (-1,2) {};
    \vertex (3) at (1,2) {};
    \vertex (4) at (2,2) {};
    \vertex (5) at (-2,-2) {};
    \vertex (6) at (-1,-2) {};
    \vertex (7) at (1,-2) {};
    \vertex (8) at (2,-2) {};
    \vertex (9) at (0,1){};
    \vertex (10) at (0,0) {};
    \vertex (11) at (-2,3) {};
    \vertex (12) at (1,3) {};
    \vertex (13) at (-1.5,-3) {};
    \vertex (14) at (1.5,-3) {};
    \vertex (15) at (-1,3) {};
    \vertex (16) at (2,3) {};
   \path[-]
    (1) edge (2)
    (2) edge (9)
    (3) edge (9)
    (3) edge (4)
    (4) edge (10)
    (1) edge (10)
    (10) edge (7)
    (9) edge (5)
    (5) edge (6)
    (7) edge (8)
    (9) edge (8)
   (10) edge (6)
   (1) edge (11)
   (2) edge (15)
   (3) edge (12)
   (4) edge (16)
   (12) edge (16)
   (5) edge (13)
   (6) edge (13)
   (7) edge (14)
   (8) edge (14)
   (11) edge (15)
    ;
\end{tikzpicture}
\\
\hline
\begin{tikzpicture}    
    \vertex (1) at (3,1) {};
    \vertex (2) at (2,1) {};
    \vertex (3) at (3,2) {};
    \vertex (4) at (3,0) {};
    \vertex (5) at (4,2) {};
    \vertex (6) at (4,0) {};
    \vertex (7) at (4,1) {};
    \vertex (8) at (5,2) {};
    \vertex (9) at (5,1) {};
    \vertex (10) at (5,0) {};
    \path[-]
    (1) edge (2)
    (2) edge (3)
    (2) edge (4)
    (3) edge (5)
    (4) edge (6)
    (1) edge (7)
    (5) edge (8)
    (7) edge (9)
    (6) edge (10)
     ;
     \end{tikzpicture}
     &
    \begin{tikzpicture}[xscale=0.4,yscale=0.4]
    \vertex (1) at (1,2) {};
    \vertex (2) at (2,1) {};
    \vertex (3) at (3,2) {};
    \vertex (4) at (4,1) {};
    \vertex (5) at (5,2) {};
    \vertex (6) at (1,-1){};
    \vertex (7) at (2,0) {};
    \vertex (8) at (3,-1){};
    \vertex (9) at (4,0) {};
    \vertex (10) at (5,-1) {};
    \vertex (11) at (3,3) {};
    \vertex (12) at (3,-2){};
    \vertex (13) at (3,4) {};
    \vertex (14) at (3,-3) {};
    \vertex (15) at (0,3) {};
    \vertex (16) at (6,-2) {};
    \vertex (17) at (6,3) {};
    \vertex (18) at (0,-2) {};
     \path[-]
    (1) edge (2)
    (2) edge (3)
    (3) edge (4)
    (4) edge (5)
    (6) edge (7)
    (7) edge (8)
    (8) edge (9)
    (9) edge (10)
    (2) edge (7)
    (4) edge (9)
    (3) edge (11)
    (8) edge (12)
    (11) edge (13)
    (12) edge (14)
    (1) edge (15)
    (10)edge (16)
    (5) edge (17)
    (6) edge (18)
;
\end{tikzpicture}
     &
 \begin{tikzpicture}
    \vertex (1) at (0,0) {};
    \vertex (2) at (1,0) {};
    \vertex (3) at (2,0) {};
    \vertex (4) at (3,0.5) {};
    \vertex (5) at (3,-0.5) {};
    \vertex (6) at (0,1) {};
    ;
    \path[-]
    (1) edge (2)
    (2) edge (3)
    (3) edge (4)
    (3) edge (5)
    (1) edge (6)
    ;
\end{tikzpicture}
&
\begin{tikzpicture}
    \vertex (2) at (1,0) {};
    \vertex (3) at (2,0) {};
    \vertex (4) at (3,0.5) {};
    \vertex (5) at (3,-0.5) {};
    \vertex (6) at (4,-0.5) {};
    \vertex (7) at (4,0.5) {};
    \vertex (8) at (5,0) {};
    \vertex (9) at (6,0) {};
    \vertex (10) at (1,1) {};
    \vertex (11) at (6,1) {};
    ;
    \path[-]
    (2) edge (3)
    (3) edge (4)
    (3) edge (5)
    (5) edge (6)
    (7) edge (4)
    (6) edge (8)
    (7) edge (8)
    (8) edge (9)
    (9) edge (11)
    (2) edge (10)
    ;
\end{tikzpicture}
\\
\hline
\begin{tikzpicture}
    \vertex (1) at (1,1) {};
    \vertex (2) at (2,1) {};
    \vertex (3) at (3,1) {};
    \vertex (4) at (4,2) {};
    \vertex (5) at (4,0) {};
    \vertex (6) at (0,2) {};
    \vertex (7) at (0,0) {};
    \vertex (8) at (4,1) {};
    \vertex (9) at (0,1) {};
    \path[-]
    (1) edge (2)
    (2) edge (3)
    (3) edge (4)
    (3) edge (5)
    (1) edge (6)
    (1) edge (7)
    (1) edge (9)
    (3) edge (8)
    ;
\end{tikzpicture}
&
\begin{tikzpicture}[yscale=0.9]
    \vertex (1) at (1.5,2) {};
    \vertex (2) at (4,2) {};
    \vertex (3) at (1,1) {};
    \vertex (4) at (1.5,1) {};
    \vertex (5) at (2,1) {};
    \vertex (6) at (3.5,1) {};
    \vertex (7) at (4.5,1) {};
    \vertex (8) at (4,1) {};
    \vertex (9) at (1.5,-1) {};
    \vertex (10) at (4,-1) {};
    \vertex (11) at (1,0) {};
    \vertex (12) at (1.5,0) {};
    \vertex (13) at (2,0) {};
    \vertex (14) at (3.5,0) {};
    \vertex (15) at (4,0) {};
    \vertex (16) at (4.5,0) {};
    \path[-]
    (1) edge (2)
    (1) edge (3)
    (1) edge (4)
    (1) edge (5)
    (2) edge (6)
    (2) edge (7)
    (2) edge (8)
    (9) edge (10)
    (9) edge (11)
    (9) edge (13)
    (9) edge (12)
    (10) edge (14)
    (10) edge (15)
    (10) edge (16)
    (3) edge (12)
    (3) edge (13)
    (4) edge (11)
    (4) edge (13)
    (5) edge (11)
    (5) edge (12)
    (6) edge (15)
    (6) edge (16)
    (8) edge (14)
    (8) edge (16)
    (7) edge (14)
    (7) edge (15)
    ;
\end{tikzpicture}
&
\begin{tikzpicture} 
    \vertex (1) at (0,0) {};
    \vertex (2) at (1,0) {};
    \vertex (3) at (1,1){};
    \vertex (4) at (0,1) {};
    \vertex (5) at (1.5,1.5) {};
    \vertex (6) at (2,2.5) {};
    \vertex (7) at (2.5,3) {};
    \vertex (8) at (2.5,2) {};
    \path[-]
    (1) edge (2)
    (2) edge (3)
    (4) edge (3)
    (1) edge (4)
    (3) edge (5)
    (5) edge (6)
    (6) edge (7)
    (6) edge (8)
     ;
     \end{tikzpicture}
     &
\begin{tikzpicture}[yscale=0.7]
    \vertex (1) at (0,0) {};
    \vertex (2) at (1,0) {};
    \vertex (3) at (1,1){};
    \vertex (4) at (0,1) {};
    \vertex (5) at (1.5,1){};
    \vertex (6) at (3,0) {};
    \vertex (7) at (3.5,1) {};
    \vertex (8) at (4.5,1) {};
    \vertex (9) at (3.5,0) {};
    \vertex (10) at (4.5,0){};
    \vertex (11) at (1.5,2){};
    \vertex (12) at (3,2) {};
    \vertex (13) at (1,3){};
    \vertex (14) at (2,2.5){};
    \vertex (15) at (2.5,2.5){};
    \vertex (16) at (3,3) {};
    \path[-]
    (1) edge (2)
    (2) edge (3)
    (4) edge (3)
    (1) edge (4)
    (7) edge (8)
    (8) edge (10)
    (9) edge (10)
    (7) edge (9)
    (3) edge (5)
    (6) edge (2)
    (7) edge (5)
    (12) edge (6)
    (11) edge (5)
    (9) edge (6)
    (11) edge (13)
    (11) edge (14)
    (12) edge (15)
    (12) edge (16)
    (14) edge (15)
    (13) edge (16)
     ;
     \end{tikzpicture}
     \\
     \hline
     
\end{tabular}}
\caption{}
\label{table:more}
\end{table}
\end{document}